\documentclass[letterpaper, 11pt]{article}
\usepackage{amsmath, amsthm, amssymb}
\usepackage{graphicx}
\usepackage{xcolor}
\usepackage{algorithm}
\usepackage{algpseudocode}
\usepackage[colorlinks,linkcolor=blue,citecolor=blue,urlcolor=magenta,linktocpage,plainpages=false]{hyperref}
\usepackage{caption}
\usepackage{subcaption}
\usepackage{setspace}
\usepackage[left=1in, right=1in, top=1in, bottom=1in]{geometry}
\usepackage{authblk}
\usepackage{changepage}
\usepackage{enumitem}

\usepackage{tikz}       
\usetikzlibrary{calc}
\usetikzlibrary{plotmarks}
\usetikzlibrary{backgrounds}
\usetikzlibrary{arrows.meta,positioning}
\usetikzlibrary{decorations.pathreplacing}
\usetikzlibrary{patterns}
\usepackage{xcolor}

\usepackage{booktabs}
\usepackage[normalem]{ulem}

\newtheorem{theorem}{Theorem}

\newtheorem{claim}{Claim}

\newtheorem{observation}{Observation}
\newtheorem{example}{Example}
\newtheorem{assumption}{Assumption}

\def\conv{\textup{conv}}
\definecolor{darkgreen}{RGB}{0, 128, 0}
\definecolor{orange}{RGB}{252, 138, 0}

\title{Aggregation of Bilinear Bipartite Equality Constraints and its Application to Structural Model Updating Problem}
\date{}
\author[1]{Santanu S. Dey\thanks{\href{mailto:santanu.dey@isye.gatech.edu}{santanu.dey@isye.gatech.edu}}}
\author[2]{Dahye Han\thanks{\href{mailto:dahye.han@gatech.edu}{dahye.han@gatech.edu}}}
\author[3]{Yang Wang\thanks{\href{mailto:yang.wang@ce.gatech.edu}{yang.wang@ce.gatech.edu}}}

\affil[1,2]{School of Industrial and Systems Engineering, Georgia Institute of Technology}
\affil[3]{School of Civil and Environmental Engineering, Georgia Institute of Technology}

\begin{document}
\maketitle

\begin{abstract}
In this paper, we study the strength of convex relaxations obtained by convexification of aggregation of constraints for a set $S$ described by two bilinear bipartite equalities. 
\textit{Aggregation} is the process of rescaling the original constraints by scalar weights and adding the scaled constraints together. 
It is natural to study the aggregation technique as it yields a single bilinear bipartite equality whose convex hull is already understood from previous literature. 
On the theoretical side, we present sufficient conditions when $\conv(S)$ can be described by the intersection of convex hulls of a finite number of aggregations, examples when $\conv(S)$ can only be obtained as the intersection of the convex hull of an infinite number of aggregations, and examples when $\conv(S)$ cannot be achieved exactly from the process of aggregation. Computationally, we explore different methods to derive aggregation weights in order to obtain tight convex relaxations. We show that even if an exact convex hull may not be achieved using aggregations, including the convex hull of an aggregation often significantly tightens the outer approximation of $\conv(S)$. 
Finally, we apply the aggregation method to obtain convex relaxation for the structural model updating problem and show that this yields better bounds within a branch-and-bound tree as compared to not using aggregations.
\end{abstract}

\section{Introduction}
In this paper, we are interested in finding {good} convex relaxations of a bounded set with two bilinear bipartite equality constraints. The general form of a set with two bilinear bipartite constraints and bounded variables is:
\begin{equation}
\begin{aligned}
\label{eq:bb_set}
    S := \left\{ x \in [0,1]^{n_1}, y \in [0,1]^{n_2} \; \left| \; x^\top Q_k y + a_k^\top x + b_k^\top y + c_k = 0, \; k \in \{1,2\}\right.\right\},
\end{aligned}    
\end{equation}
where $Q_k \in \mathbb{R}^{n_1 \times n_2}, a_k \in \mathbb{R}^{n_1}, b_k \in \mathbb{R}^{n_2}, c_k \in \mathbb{R}$ for all $k \in \{1, 2\}.$
The term bilinear refers to the fact that the constraints are described using quadratic functions where all the degree two terms involve products of two distinct variables.
The set \eqref{eq:bb_set} is referred to as bipartite because variables can be divided into two groups, $x$ and $y$, where no degree two terms include variables from the same group. 
Note that any problem with different variable bounds can always be scaled so that the variable bounds are $[0,1]$, and any inequality constraint can be converted to an equality constraint by adding a slack variable, which is also bounded since other variables are bounded.
While the set $S$ is quite general, we will apply our results to an important application from the structural engineering called the finite element model (FEM)  updating problem~\cite{otsuki2021finite} where finding good convex relaxation of $S$ can improve the performance of the branch-and-bound process.
\subsection{Aggregations}
In the papers \cite{dey2019new,santana2020convex}, the authors show that for a set described by a single bilinear bipartite equality constraint with bounded variables, the convex hull can be obtained using a disjunctive procedure yielding a second-order cone representable set. In particular, the convexification process of a bilinear bipartite equality constraint involves finding convex hulls of sets obtained by fixing all but one $x$ variable and one $y$ variable to their bounds, and then taking the convex hull of all these resulting convex sets.

Since we understand how to find a convex hull of a bounded set with a one-row bilinear bipartite equality constraint, 
in order to obtain good convex relaxations for (\ref{eq:bb_set}),
a natural approach is to consider aggregation of the two constraints to obtain a one-row relaxation and then convexifying it. Formally, an aggregation takes the following form,
\begin{eqnarray*}
S_{\lambda} := \left\{ x\in [0,1]^{n_1}, y\in [0,1]^{n_2} \left | \,  \begin{array}{rcl} 
    \vspace{0.2cm}\lambda_1 \cdot \left( x^\top Q_1 y + a_1^\top x + b_1^\top y + c_1\right) \\ 
    + \; \lambda_2 \cdot \left( x^\top Q_2 y + a_2^\top x + b_2^\top y + c_2 \right) &=& 0
\end{array} \right. \right\},
\end{eqnarray*}
where $\lambda = (\lambda_1, \lambda_2) \in \mathbb{R}^2$. 
Now, $S_\lambda$ is a bounded set with a single bilinear bipartite equality constraint. We use $\conv(S)$ to denote the convex hull of set $S$.

In this paper, the main object of study is the set, 
\begin{eqnarray}
\label{eq:aggregation}
    \bigcap_{\lambda \in T} \conv(S_{\lambda}),
\end{eqnarray} 
which is referred to as the \emph{aggregation closure} of $S$ over $T \subseteq \mathbb{R}^2$.
We note that since the aggregation closure is the intersection of convex sets, it is also convex. Moreover, since $S \subseteq S_\lambda$ for any $\lambda \in \mathbb{R}^2$, we obtain that $\conv(S) \subseteq \conv(S_\lambda)$, thus implying that $\conv(S) \subseteq \bigcap_{\lambda \in T} \conv(S_\lambda)$ for any $T \subseteq \mathbb{R}^2$. 
For the rest of the paper, we wish to evaluate the strength of aggregation closure (\ref{eq:aggregation}) of $S$, both theoretically and computationally, to see how it can give a good relaxation for (\ref{eq:bb_set}), which is a key element in providing dual bounds for globally solving non-convex problems within a branch-and-bound tree framework. 

The idea of finding a convex hull or valid inequalities for the convex hull of an aggregation has been successfully used in integer programming. For example, the classical paper~\cite{marchand2001aggregation} studies how to aggregate inequalities before finding mixed integer rounding inequalities from the aggregated constraint, and the Gomory mixed-integer cuts can also be viewed as generated from an aggregation of the original constraints of a MILP~\cite{cornuejols2001elementary}. The paper~\cite{bodur2018aggregation} studies the strength of aggregation closure for packing and covering mixed integer linear sets. See~\cite{dey2018theoretical} for a discussion on several classes of cutting-planes developed using aggregations in mixed integer linear programming.

A closely related idea to the aggregation closure is the idea of directly using a relaxation $S_{\lambda}$ for some $\lambda$, with the underlying motivation being that $S_{\lambda}$ is a much simpler non-convex set to optimize over rather than the original set. 
This relaxation is referred to as the \textit{surrogate} relaxation. 
The concept of surrogate constraint was first introduced by \cite{glover1965multiphase} in the context of the 0-1 integer programming. The authors of \cite{balas1967discrete} and \cite{geoffrion1969improved} showed how this constraint can give a strong relaxation under certain conditions. Later, authors of \cite{greenberg1970surrogate} generalized the result beyond the 0-1 integer programming. 
Recently, an extensive computational evaluation of the surrogate dual for general mixed integer nonlinear programs was conducted in \cite{muller2022generalized}.

Since aggregating with different weights provides different relaxations that outer approximate the given set, finding a ``good'' aggregation is essential. Particularly, the question of whether an intersection of aggregated constraints can represent a convex hull has been studied for quadratic inequalities. The paper \cite{yildiran2009convex} showed that the convex hull of two quadratic inequalities is given by at most two aggregated inequalities. The paper \cite{dey2022obtaining} expands these results to the case of three quadratic inequalities. Finally, \cite{blekherman2024aggregations} presents more general sufficient conditions when aggregations can yield a convex hull of a set defined by quadratically constrained inequalities, and shows specific conditions when finitely many aggregations would suffice to generate the convex hull.  We note here that all these results are for a direct intersection of aggregations without taking convex hulls, and they do not hold when bounds are added to the variables. 
Therefore, these results are not directly applicable to the case of aggregation closure (\ref{eq:aggregation}) for the bounded set $S$.

\subsection{Contributions of this paper}
The main contributions of this paper are driven by the following questions:
\begin{enumerate}
    \item Can $\conv(S)$ be represented with an intersection of \textit{finite} number of $\conv(S_\lambda)$? We show that for a special case when $n_1 = n_2 = 1$ (that is, $x$ and $y$ are one-dimensional), $\conv(S)$ can be obtained with the intersection of the convex hull of at most three aggregated constraints. However, for $n_1 + n_2 \geq 3$, we show a counterexample where an infinite number of aggregations, that is the full strength of the aggregation closure with all the values of $\lambda$s ($T = \mathbb{R}^2$), are required to obtain the convex hull. 
    \item We then ask if the above result can be generalized, so that the intersection of possibly \textit{infinitely} many aggregations always provides the convex hull. We answer this question in the negative, by showing another counterexample with $n_1+n_2=3$ variables where even an intersection of infinitely many aggregations does not result in the convex hull of the original set.
    \item If aggregations cannot produce an exact convex hull, can it still be useful to tighten the feasible region? We explore different aggregation techniques and experiment on randomly generated bilinear instances as well as instances from real applications. We show that indeed aggregations can still be powerful in providing much tighter convex relaxation. When used to build convex relaxations within a branch-and-bound tree, the aggregation approach leads to improved dual bounds for the finite element model (FEM) updating problem.
\end{enumerate}

\paragraph{Notation and organization of the paper} 
Given a positive integer $n$, we let $[n]:=\{1,2,\ldots,n\}$.
For a countable set $T$, we use $|T|$ to denote the cardinality of the set.
The rest of the paper is organized as follows. In Section~\ref{sec:theoretical_evaluation} we present a list of results that serve as a theoretical evaluation of the aggregation procedure, followed by Section~\ref{sec:computation_evaluation} where we present a computational evaluation of the aggregation procedure.  This section includes experiments to evaluate how to find aggregation weights for the constraints in practice, and also the application of the resulting method to the FEM update problem. Section~\ref{sec:conclusion} presents our conclusions. Section~\ref{sec:proof_n1=n2=1} to Section~\ref{sec:agg_no_work_proof} provides proofs of the results presented in Section \ref{sec:theoretical_evaluation}.

\section{Theoretical evaluation of aggregation}
\label{sec:theoretical_evaluation}
Throughout this section, we make two assumptions about the set $S$ in \eqref{eq:bb_set}.
\begin{assumption}
\label{assumption:1}
$S$ is a nonempty set.
\end{assumption}
\begin{assumption}
\label{assumption:2}
    Two constraints are independent of each other so that there exists no $\lambda \in \mathbb{R}$ such that $x^T Q_1 y + a_1^\top x + b_1^\top y + c_1 = \lambda (x^T Q_2 y + a_2^\top x + b_2^\top y + c_2)$.
\end{assumption}
If $S$ does not satisfy Assumption~\ref{assumption:2}, then only a single bilinear bipartite equality constraint is necessary to describe the set instead of two. 

Our first main result is the following sufficient condition when the aggregation closure yields the convex hull.
\begin{theorem}
\label{theorem:n1=n2=1}
Consider the set $S$ described in (\ref{eq:bb_set}) with $n_1 = n_2 = 1$. Then there exists $T \subseteq \mathbb{R}^2$ where $|T| \leq 3$ such that:
\begin{align*}
    \conv(S) = \bigcap_{\lambda \in T} \conv(S_\lambda).
\end{align*}
\end{theorem}
Our proof of Theorem~\ref{theorem:n1=n2=1} is provided in Section~\ref{sec:proof_n1=n2=1}. The proof considers several cases based on the structure of the set $S$. For example, $\textup{conv}(S)$ can either be a point or a line segment. Within each of these cases, there are several sub-cases based on whether the quadratic constraints define non-degenerate hyperbola or union of lines. Our proof is constructive, that is, it explicitly provides recipes of the three or fewer aggregations needed to obtain the convex hull for each case.   
We provide an example below to further illustrate this result.
\begin{example}
\label{ex:n1=n2=1}
Consider the set below:
\begin{eqnarray}
S = \left\{ x, y \in [0, 1]^2\, \left | \,  \begin{array}{l} (x+0.5)y = 0.5\\ (x-1)(y+1.5) = -1 \end{array} \right. \right\}.
\end{eqnarray}
The only point that satisfies both constraints in the $[0,1]^2$ box is $(x,y)=(0.5,0.5)$; hence, $S$ is a singleton and $\conv(S) = \{(0.5,0.5)\}$.
One convex approximation suggested in \cite{dey2019new} is taking an intersection of the convex hull of each constraint. However, note that this does not yield an exact convex hull of $S$ as illustrated in Figures~\ref{figure:n1=n2=1_a} and \ref{figure:n1=n2=1_b}. Now consider the following aggregation with $\lambda = (1, -1)$:
$$S_{\lambda} := \left\{ x, y \in [0, 1]^2\, \left | \, (xy + 0.5y - 0.5) - (xy  +1.5 x - y - 0.5) = -1.5x + 1.5 y = 0 \right. \right\}.$$
$S_\lambda$ is a straight line passing through $(0.5,0.5)$ 
and 
$$\textup{conv}(S) =\textup{conv}(S_{(1,0)}) \ \cap \ \textup{conv} (S_{(0,1)}) \ \cap \ \textup{conv}(S_{(1,-1)}),$$
as shown in Figure~\ref{figure:n1=n2=1_c}.
\end{example}
\begin{figure}[hbt!]
\centering
\begin{minipage}{0.32\textwidth}
    \centering
    \begin{tikzpicture}[scale=2.5]
    \draw[dashed] (1,0) -- (1,1);
    \draw[dashed] (0,1) -- (1,1);
    \draw[-] (0.25,0.25) -- (0.25,0.75);
    \draw[-] (0.25,0.25) -- (0.75,0.25);
    \draw[-] (0.75,0.25) -- (0.75,0.75);
    \draw[-] (0.25,0.75) -- (0.75,0.75);
    \draw[->] (-0.1,0) -- (1.2,0) node[right] {$x$};
    \draw[->] (0,-0.1) -- (0,1.2) node[above] {$y$};
    \draw[red,domain=0:1,line width=1pt] plot (\x,{0.5/(\x+0.5)});
    \draw[blue,domain=1/3:0.6,line width=1pt] plot (\x,{-1.5-1/(\x-1)});
    \fill (0.5, 0.5) circle[radius=0.8pt];
    \begin{scope}
      \fill[pattern=north east lines,pattern color=red, opacity=0.5, domain=0:1, variable=\x]
        (0,1) plot (\x,{0.5/(\x+0.5)}) -- plot (\x,{-2/3*\x+1}) -- (1,1/3) -- cycle;
    \end{scope}
    \begin{scope}
      \fill[pattern=north west lines,pattern color=blue, opacity=0.5, domain=1/3:0.6, variable=\x]
        (1/3,0) plot (\x,{-1.5-1/(\x-1)}) -- plot (\x,{15/4*\x-5/4}) -- (0.6,1) -- cycle;
    \end{scope}
    \end{tikzpicture}
    \subcaption{\footnotesize Convex hull of each constraint}
    \label{figure:n1=n2=1_a}
\end{minipage}%
\begin{minipage}{0.32\textwidth}
    \centering
    \begin{tikzpicture}[scale=5]
    \draw[-] (0.25,0.25) -- (0.25,0.75);
    \draw[-] (0.25,0.25) -- (0.75,0.25);
    \draw[-] (0.75,0.25) -- (0.75,0.75);
    \draw[-] (0.25,0.75) -- (0.75,0.75);
    \draw[->,opacity=0] (0.2,0.25) -- (0.85,0.25) node[right] {$x$};
    \draw[->,opacity=0] (0.25,0.2) -- (0.25,0.85) node[above] {$y$};
    \draw[red,domain=0.25:0.75,line width=1pt] plot (\x,{0.5/(\x+0.5)});
    \draw[blue,domain=3/7:5/9,line width=1pt] plot (\x,{-1.5-1/(\x-1)});
    \fill (0.5, 0.5) circle[radius=0.6pt];
    \begin{scope}
      \fill[pattern=north east lines,pattern color=red, opacity=0.5, domain=0.25:0.75, variable=\x]
        (0.25,2/3) plot (\x,{0.5/(\x+0.5)}) -- (0.75, 0.4) -- (0.75, 0.5) -- plot (\x,{-2/3*\x+1}) -- (0.375, 0.75) -- (0.25,0.75) -- cycle;
    \end{scope}
    \begin{scope}
      \fill[pattern=north west lines,pattern color=blue, opacity=0.5, domain=0.4:5/9, variable=\x]
        (0.4,0.25) -- (3/7, 0.25) -- plot (\x,{-1.5-1/(\x-1)}) -- (5/9, 0.75) -- (8/15, 0.75) -- plot (\x,{15/4*\x-5/4}) -- cycle;
    \end{scope}
    \end{tikzpicture}
    \subcaption{\footnotesize Zoom-in of Figure~\ref{figure:n1=n2=1_a}}
    \label{figure:n1=n2=1_b}
\end{minipage}%
\begin{minipage}{0.32\textwidth}
    \centering
    \begin{tikzpicture}[scale=2.5]
    \draw[dashed] (1,0) -- (1,1);
    \draw[dashed] (0,1) -- (1,1);
    \draw[->] (-0.1,0) -- (1.2,0) node[right] {$x$};
    \draw[->] (0,-0.1) -- (0,1.2) node[above] {$y$};
    \draw[red,domain=0:1,line width=1pt] plot (\x,{0.5/(\x+0.5)});
    \draw[blue,domain=1/3:0.6,line width=1pt] plot (\x,{-1.5-1/(\x-1)});
    \draw[black,domain=0:1,line width=1pt] plot (\x,{\x});
    \fill (0.5, 0.5) circle[radius=0.8pt];
    \begin{scope}
      \fill[pattern=north east lines,pattern color=red, opacity=0.5, domain=0:1, variable=\x]
        (0,1) plot (\x,{0.5/(\x+0.5)}) -- plot (\x,{-2/3*\x+1}) -- (1,1/3) -- cycle;
    \end{scope}
    \begin{scope}
      \fill[pattern=north west lines,pattern color=blue, opacity=0.5, domain=1/3:0.6, variable=\x]
        (1/3,0) plot (\x,{-1.5-1/(\x-1)}) -- plot (\x,{15/4*\x-5/4}) -- (0.6,1) -- cycle;
    \end{scope}
    \end{tikzpicture}
    \subcaption{\footnotesize $\conv(S)$ achieved}
    \label{figure:n1=n2=1_c}
\end{minipage}%
\caption{Each constraint is represented with red and blue curves in the $[0,1]^2$ box and their respective convex hulls are shaded areas of the same color. The first subfigure shows that the intersection of the convex hull of each constraint does not yield a convex hull of $S$. The second figure zooms in the middle box portion of the first subfigure and shows that there exists a nonempty area where the red-shaded area and the blue-shaded area overlap beyond $\conv(S) = \{(0.5,0.5)\}$. In the last subfigure, we see that adding the aggregation $\textup{conv}(S_{(1,-1)})$ allows us to achieve $\conv(S)$ by the aggregation closure.}
\end{figure}
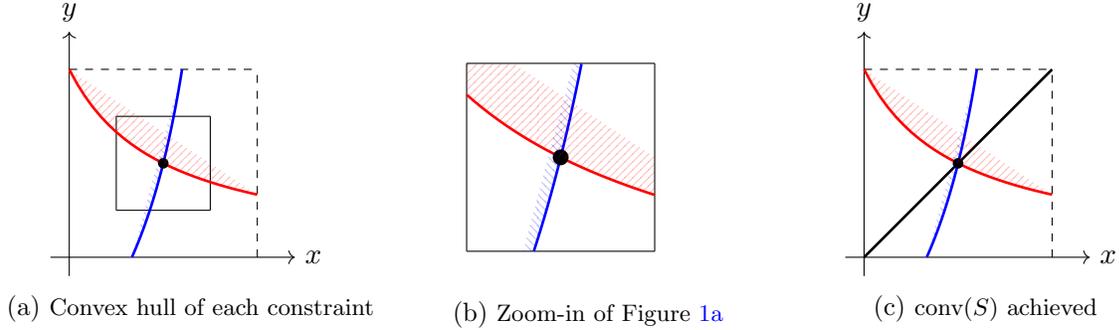

Based on the result of Theorem~\ref{theorem:n1=n2=1}, we might conjecture that the intersection of the convex hull of finitely many aggregated sets may represent the convex hull for a more general case of $n_1$ and $n_2$. However, as soon as we increase the number of bilinear terms by adding just one more variable, we immediately find an example where infinite aggregations are required to yield the convex hull.

\begin{theorem}
\label{theorem:infinite_agg_needed}
There exists an instance of set $S$ as described in (\ref{eq:bb_set}) with $n_1 + n_2 \geq 3$, such that infinite aggregation is needed to obtain the convex hull. In other words, there exists $S$ such that 
$$\textup{conv}(S) = \bigcap_{\lambda \in \mathbb{R}^2} \textup{conv}(S_\lambda),$$ and
$$\textup{conv}(S) \subsetneq \bigcap_{\lambda \in T} \textup{conv}(S_\lambda),$$
where $T \subseteq \mathbb{R}^2$ and $|T| < \infty.$
\end{theorem}
Our proof of Theorem~\ref{theorem:infinite_agg_needed} is provided in Section~\ref{sec:infinite_agg_needed_proof}. The proof of this result uses the simple set 
$$S= \left\{ x, y_1, y_2 \in [0, 1]^3 \, \left| \, \begin{array}{c} xy_1 = 0.5 \\ xy_2 = 0.5\end{array} \right. \right\}.$$

It it is easy to see that $\textup{conv}(S)$ is contained in the hyperplane $\{x, y_1, y_2\,|\, y_1 = y_2\}$. It turns out that $\bigcap_{\lambda \in T} \textup{conv}(S_\lambda)$ is not contained in this hyperplane for any finite set $T$. On the other hand, the full power of aggregation closure, that is using all aggregation weights, achieves the convex hull for this set.

Can we always expect to obtain the convex hull with the intersection of infinitely many aggregated sets? The next result shows that this is not a sufficient condition through a counterexample. It may seem natural that such a set exists that infinite aggregations do not yield the convex hull, but it was not straightforward to find such an example.
\begin{theorem}
\label{theorem:agg_no_work}
There exists an instance of the set $S$ as described in (\ref{eq:bb_set}) with $n_1 + n_2 \geq 3$, such that  the aggregation closure over $\mathbb{R}^2$ is not equal to the convex hull of $S$, that is,
\begin{align*}
    \conv(S)  \subsetneq \bigcap_{\lambda \in \mathbb{R}^2} \conv(S_\lambda).
\end{align*}
\end{theorem}
A set that serves as an example for our proof of Theorem~\ref{theorem:agg_no_work} is the following:
\begin{eqnarray*}
S = \left\{ x, y_1, y_2 \in [0, 1]^3\, \left | \,  
\begin{array}{lcl} 
-2 x y_1 + 9 x y_2 + y_1 -5 y_2 &=& 0 \\ 
5 x y_1 + 3 y_1 + 3 y_2 &= & 6 
\end{array} \right. \right\}.
\end{eqnarray*}
The proof of Theorem~\ref{theorem:agg_no_work} relies on finding a point $(x^*, y_1^*, y_2^*) \notin \conv(S)$ but $(x^*, y_1^*, y_2^*) \in \conv(S_\lambda)$ for any $\lambda \in \mathbb{R}^2$. Certifying that $(x^*, y_1^*, y_2^*) \in \conv(S_\lambda)$ involves explicitly finding points in $S_{\lambda}$ and convex combination weights of these points to obtain $(x^*, y_1^*, y_2^*)$. The entire details of our proof are presented in Section~\ref{sec:agg_no_work_proof}.

\section{Computational evaluation of aggregation}
\label{sec:computation_evaluation}
As presented in Section~\ref{sec:theoretical_evaluation}, the exact convex hull of a set $S$ described in (\ref{eq:bb_set}) is given by the intersection of convex hulls of aggregations in some specific cases. In this section, we would like to evaluate the quality of bounds given by the aggregation closure even if the exact convex hull cannot be obtained using it. In order to achieve this goal we will present two computational experiments on: (i) randomly generated instances; and (ii) instances from the finite element model updating problem. 
Before we discuss the results of our experiments, we discuss different methods to find aggregation weights, convex relaxations to be compared, and what metric we use to evaluate the strength of the aggregation.

\subsection{Finding aggregation weights}
\label{sec:aggregation_weights}
A key part of adding the convex hull of an aggregated constraint is finding a ``good'' way to aggregate, that is, finding appropriate weights. Although adding convex hulls of many different aggregations gives a tighter outer approximation of the convex hull of the two constraints, that is 
$$ \cap_{\lambda \in T_1} \conv(S_\lambda) \subseteq \cap_{\lambda \in T_2} \conv(S_\lambda) \ \textup{ for any }T_2 \subseteq T_1,$$ there is a trade-off with the size and complexity of the problem. We therefore consider adding the convex hull of exactly one aggregated constraint for two constraints. To find aggregation weights for the two constraints, we test several heuristic methods. Two of these heuristics are based on the spirit of separating a relaxed solution $(\hat x, \hat y)$. A relaxed solution can be easily obtained by solving a relaxation problem without aggregation. 
 Given the similarity between the aggregation approach and the surrogate relaxation approach, the third method follows the approach used in surrogate duality theory.
\begin{enumerate}
    \item Grid search: 
    We search for aggregation weights $\lambda$ from a grid $G$. In our experiments, we set the size of $G$ to be 20, where    
    \begin{align*}
        G = \{(1,2),(1,2^2),...,(1,2^5),(1,-2),...,(1,-2^5),(2,1),...,(2^5,1), (-2,1),...,(-2^5,1)\}.
    \end{align*}
    We pick $\lambda \in G$ that maximizes the distance between $\conv(S_\lambda)$ and a relaxed solution $(\hat x, \hat y)$ that we would like to separate. Specifically, define:
    \begin{equation}
    \begin{aligned}
    \label{eq:gridsearch}
        d(\lambda) = \quad \min_{x,y} \quad & ||(x,y) - (\hat x, \hat y)|| \\
        s.t. \quad & (x, y) \in \conv(S_\lambda) \\
        & x \in [0,1]^{n_1}, y \in [0,1]^{n_2}.
    \end{aligned}
    \end{equation}
    Pick $\lambda^* = \arg\max\{d(\lambda), \lambda \in G\}$.
    
    \item Simple search:  
    Given a relaxed solution $(\hat x, \hat y)$, we first fix $y$ to $\hat y$ in $S_\lambda$. Then, $S_\lambda |_{y=\hat y}$ is a hyperplane in the $x$ space with parameters defined by $\lambda$. The resulting hyperplane can be written in the form of $p^\top x = q$ for some $p$ and $q$. We wish to pick $\lambda \in \mathbb{R}^2$ such that it maximizes the distance between the hyperplane and $\hat x$. 
    This gives us the following optimization problem where     $p^\top \hat x - q$ is the distance between the hyperplane and $\hat x$. 
    \begin{equation}
    \begin{aligned}
    \label{eq:simplesearch}
        d(\hat y) = \max_{\lambda_1, \lambda_2, p, q} \quad 
        & p^T \hat x - q \\
        & p = \lambda_1 Q_1 \hat y + \lambda_1 a_1 + \lambda_2 Q_2 \hat y + \lambda_2 a_2 \\
        & q = - \left(\lambda_1 b_1^\top \hat y + \lambda_1 c_1 + \lambda_2 b_2^\top \hat y + \lambda_2 c_2 \right) \\
        & ||p||_2 \leq 1, -100 \leq \lambda_1, \lambda_2 \leq 100
    \end{aligned}
    \end{equation}
    We can similarly define $d(\hat x)$ by fixing $x$ to $\hat x$ and find aggregation weights that maximize the distance between a hyperplane defined with respect to $y$ and $\hat y$.
    We then pick $\lambda = (\lambda_1, \lambda_2)$ corresponding to the larger to the two $d(\hat x)$ or $d(\hat y)$.
    \item Surrogate search: We consider finding the aggregation weights based on results from the surrogate duality literature. 
    In \cite{muller2022generalized}, the authors present a Bender's decomposition-type approach to find surrogate dual multipliers that will yield the best lower bound. We adopt algorithm 1 in their paper to find our multipliers. Specifically, the algorithm is based on constraints of the form $g(x) \leq 0$. Hence, we split our equality constraints into two inequality constraints. Then the algorithm returns surrogate dual multipliers $\lambda \in \mathbb{R}^{4}_+$. We then define $\mu = (\lambda^*_1 - \lambda^*_2, \lambda^*_3 - \lambda^*_4)$ to be our aggregation weights.
    There is also a generalized version of surrogate relaxations which allows multiple aggregation weights \cite{glover1975surrogate}.
    Although multiple aggregations give a tighter relaxation, as we discussed earlier, for the purpose of this paper, we focus on adding one aggregated constraint to evaluate its performance.
\end{enumerate}

\subsection{Convex relaxation}
We will consider three types of convex relaxations in our experiments:
\begin{enumerate}
    \item McCormick relaxation: This is the standard linear programming relaxation using the McCormick relaxation of each bilinear term \cite{mccormick1976computability}. For specific types of bilinear bipartite problems, McCormick relaxation is known to perform better than semi-definite programming relaxation~\cite{gu2024solving}.
    \item One-row relaxation: We
    build a polyhedral outer approximation of the convex hull of each bilinear constraint. This polyhedral relaxation is built in the same way as in section 5.2.2 of \cite{dey2019new}. The final convex relaxation is the intersection of each of these polyhedra together with other linear constraints and bounds in the problem.
    \item Aggregation relaxation: We find aggregation weights for selected subsets of two rows, using one of the methods from Section~\ref{sec:aggregation_weights}. Then we add the polyhedral outer approximation of the convex hull of each aggregated constraint to the one-row relaxation.
\end{enumerate}

\subsection{Evaluation metric}
The main evaluation metric we use to compare the quality of the lower bound achieved is the relative gap improvement ($\Delta\rho$). We first define the relative optimality gap ($\rho$). Let $z_{opt}$ be the best primal objective value found and $z^*$ be the lower bound achieved by a method that we consider. 
For example, $z^*_{Mc}$ and $z^*_{1row}$ are the lower bounds achieved from McCormick relaxation and the one-row relaxation. For aggregation relaxation, we use $z^*_{simple}$, $z^*_{surr}$, and $z^*_{grid}$ to be the lower bound achieved from aggregation where the weights are obtained by simple heuristic technique,  surrogate search, and grid search, respectively.
Then, the relative optimality gap is defined as:
\begin{align*}
    \rho = \frac{|z_{opt}-z^*|}{|z_{opt}|} \times 100 \%.
\end{align*}
We choose one of $z^*_{Mc}$, $z^*_{1row}$, $z^*_{simple}$, $z^*_{surr}$, and $z^*_{grid}$ to be the base and calculate how much the relative gap has improved compared to the base. For example, if we use $z^*_{Mc}$ to be the base, the relative gap improvement is defined as:
\begin{align*}
    \Delta \rho_{Mc} = \frac{z^* - z^*_{Mc}}{z_{opt} - z^*_{Mc}} \times 100\% = \frac{\rho - \rho_{Mc}}{\rho_{Mc}} \times 100\%.
\end{align*}
If we use $z^*_{1row}$ to be the base, the relative gap improvement is defined as:
\begin{align*}
    \Delta \rho_{1row} = \frac{z^* - z^*_{1row}}{z_{opt} - z^*_{1row}} \times 100\% = \frac{\rho - \rho_{1row}}{\rho_{1row}} \times 100\%.
\end{align*}
Smaller $\rho$ implies that the relative optimality gap is small and larger $\Delta \rho$ implies that the relative gap improvement against the baseline is large. Hence, smaller $\rho$ and larger $\Delta \rho$ is desirable.

\subsection{Environment and software}\label{sec:software}

{All numerical instances are implemented on Julia version 1.7 with Gurobi version 10.1 used to solve the linear programming relaxation of the lower bounding problem and Ipopt~\cite{wachter2006implementation} was used to solve the upper bound problem at each node of the branch-and-bound tree. Gurobi was also used to solve \eqref{eq:gridsearch} and \eqref{eq:simplesearch}. 
BARON version 24.5.8 was used for benchmarking our results. A custom branch-and-bound was built adapting the Julia package BranchAndBound.jl \cite{bnbjulia}.} The experiments were run on a personal laptop with a Windows 64-bit operating system with a 1.8GHz processor and 16GB RAM.

\subsection{Randomized experiment}
We first test the power of aggregation on randomly generated instances.
\subsubsection{Instance generation}
We generate random instances of the form:
\begin{align*}
    \min_{x, y} \quad & f^\top x + g^\top y\\
    s.t. \quad & x^\top Q_k y + a_k^\top x + b_k^\top y + c_k = 0, \qquad k \in [2]\\
    & x \in [0,1]^{n_1}, y \in [0,1]^{n_2}.
\end{align*}
These small dense instances are generated for $(n_1,n_2) \in \{(2,2), (3,3), (5,5)\}$ with all variables appearing in each constraint. For fixed $n_1$ and $n_2$, we generate entries of $f$, $g$, $Q_k$, $a_k$, $b_k$ and $c_k$ iid from discrete uniform distribution on $\{-10,10\}$. We generated 10 instances each for every choice of $(n_1, n_2)$.

\subsubsection{Evaluation of results on random instances}
The results presented in Figure~\ref{fig:randexp_relimprovement_mc} show the relative gap improvement against the McCormick relaxation. A lot of the gap from McCormick relaxation is already closed by taking one-row relaxation. 
However, adding a single aggregated constraint further closes the gap by nontrivial amounts. The additional improvement is more than 10\% for $(n_1,n_2) \in \{(2,2),(3,3)\}$ instances and almost 8\% for $(n_1,n_2) = (5,5)$ instances on average. 

Some instances are solved close to optimality with relative optimality gap $\rho < 0.01\%$ using the one-row relaxation without any aggregation. In Figure~\ref{fig:randexp_relimprovement_1row}, we focus on the remaining instances where the one-row relaxation without the aggregation does not close a lot of the gap and highlight the impact of adding an aggregated constraint. From both figures, we conclude that the largest improvements are achieved from the grid search method. 

\begin{figure}[tbh!]
    \centering
    \includegraphics[width=0.75\linewidth]{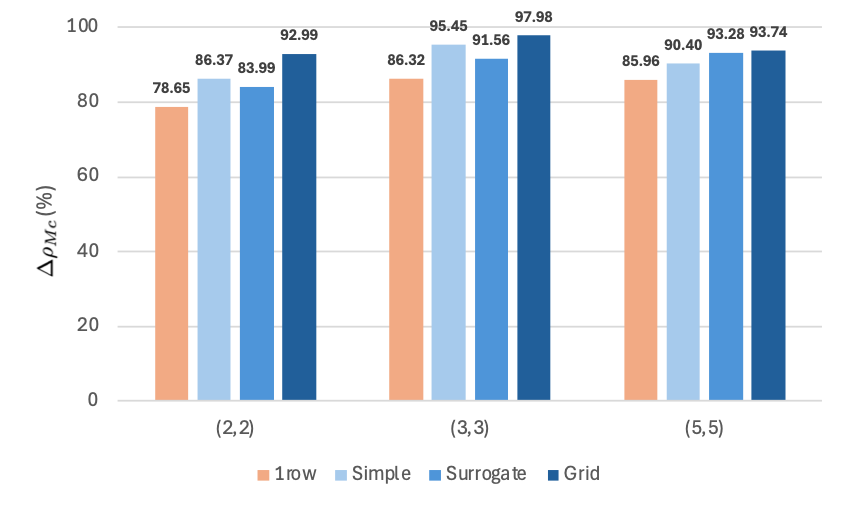}
    \caption{Average relative improvement (\%) against the gap achieved from the McCormick relaxation for different choices of $(n_1,n_2)$.}
    \label{fig:randexp_relimprovement_mc}
\end{figure}

\begin{figure}[tbh!]
    \centering
    \includegraphics[width=0.75\linewidth]{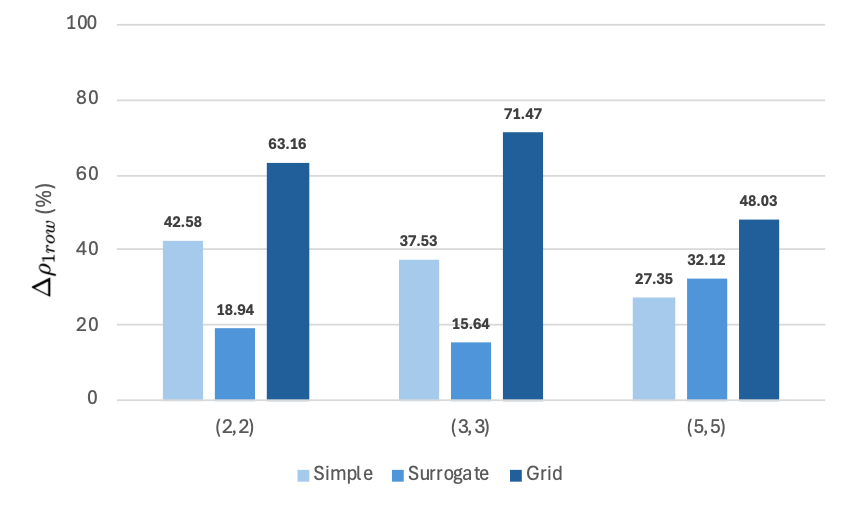}
    \caption{Average relative improvement (\%) against the gap achieved from the one-row relaxation for different choices of $(n_1,n_2)$ for instances excluding $\rho_{1row} < 0.01\%$.}
    \label{fig:randexp_relimprovement_1row}
\end{figure}

We also want to ensure that the search times for aggregation weight are reasonable. Table~\ref{tab:randexp_times} shows the average time required to find aggregation weights for different methods. Note that the simple search takes almost no time as it requires solving simple convex optimization problems. The grid search method requires solving $|G|$ number of optimization problems independently; hence, we use 8 threads to parallelize. We also report the estimated time that would take to solve if we had the computing power to parallelize solving all $|G|$ number of optimization problems in column ``Grid (parallel est.)''. We note that the time spent can be significantly reduced if all is parallelized. 
The time reported for ``Grid (parallel est.)'' column adds: (i) time to solve a relaxation problem to obtain a relaxed solution $(\hat x, \hat y)$ to separate; (ii) the maximum of the time spent to solve the optimization problem \eqref{eq:gridsearch} among $\lambda \in G$; and (iii) the time spent to select $\lambda^*$. We can hence conclude that using the grid search method is also reasonable in time.
Unlike the grid search method, the surrogate search method requires solving optimization problems recursively which hinders parallelization.

\begin{table}[tbh!]
\centering
\caption{Average time spent in finding aggregation weights (sec)}
\label{tab:randexp_times}
\begin{tabular}{lrrrr}
\toprule
$(n_1,n_2)$ & \multicolumn{1}{c}{Simple} & \multicolumn{1}{c}{Surrogate} & \multicolumn{1}{c}{Grid} & \multicolumn{1}{c}{Grid (parallel est.)} \\
\midrule
$(2,2)$ & 0.26 & 6.84 & 7.21 & 1.71 \\
$(3,3)$ & 0.63 & 9.31 & 8.36 & 2.27 \\
$(5,5)$ & 1.09 & 13.51 & 22.15 & 5.76 \\
\bottomrule
\end{tabular}
\end{table}

\subsection{Finite Element Model updating problem}
The finite element model (FEM) update problem is a crucial problem in the field of structural engineering, that seeks to minimize the differences between a structure's as-built behaviors and those predicted by its FEM.
This problem can be viewed as selecting optimal stiffness parameter values from a given affine subspace parameterized by $\alpha$ variables, such that the generalized eigenvalues ($\lambda$) between the stiff and mass matrices as well as some entries of corresponding eigenvectors ($\psi$) match experimentally observed values from vibration testing \cite{otsuki2021finite}. 
The goal is to minimize the maximum absolute difference between the experimental measurements of eigenvalues and eigenvectors and the eigen-pairs produced by the FEM. This value is represented as $\delta$ and the problem can be reformulated as below:
\begin{align}
    \min_{\alpha, \lambda, \psi, \delta} \quad & \delta \notag \\
     \text{s.t.} \quad 
    & \left( K_0 + \sum_{j \in [k]} \alpha_j K_j - \lambda_i M \right) \psi_i = 0 && \forall i \in [m] \label{eq:eigen_constraint}\\
    & (\lambda, \psi, \delta) \in P \notag \\
    & \alpha \in [\alpha_{lb}, \alpha_{ub}], \ \lambda \in [\lambda_{lb}, \lambda_{ub}], \ \psi \in [\psi_{lb}, \psi_{ub}] \notag
\end{align}
where $M, K_0, K_j \in \mathbb{R}^{n \times n}$ for $j \in [k]$ and $P$ is a polyhedron modeling the weighted penalties with respect to the deviation from experimental eigenvalues and eigenvectors.
The number $n$ representing the number of rows of the stiffness matrix $K_0$ corresponds to the 
degrees of freedom which represents the number of stories of a structure in simplified models such as the shear frame model used in this paper;
henceforth we refer to an instance with $n\times n$ matrices as an $n$-story instance. 
After rescaling and linear transformation, the problem can be equivalently formulated as a bilinear bipartite problem of the form in \eqref{eq:bb_set} where we refer to the $\alpha$ and $\lambda$ variables as $x$ variables and refer to the $\psi$ variables as $y$ variables.
We use simulated structural instances similar to the 18-story structure in \cite{otsuki2021finite}. The instances are available \href{https://github.com/han903/aggregation_bbe}{here}.

\begin{table}[tbh!]
\caption{Summary of Instances}
\label{tab:instances}
\centering
\begin{tabular}{lccc}
\toprule
    Data & \# of $x$ variables ($n_1$) & \# of $y$ variables ($n_2$) & \# of bilinear constraints ($n m$) \\
    \midrule
    12-story & 14 & 24 & 24 \\
    16-story & 18 & 48 & 48 \\
    \bottomrule
\end{tabular}
\end{table}
\subsubsection{Selection of rows for aggregation}
We note that all the $K_j$ matrices are tridiagonal. Thus, the supports of $x$ variables are very similar in the constraints corresponding to consecutive rows of \eqref{eq:eigen_constraint}. 
Therefore, we choose only to aggregate constraints corresponding to consecutive rows $\{1,2\}, \{3,4\}, \dots$, as it does not create very dense constraints. Note that convexification takes a long time if we have dense constraints as the number of disjunctions increases exponentially with the number of variables.
Overall, this implies that we add $(n/2)\cdot m$ aggregated constraints, which is an addition to the $nm$ bilinear constraints.

\subsubsection{branch-and-bound}
We apply the convexification idea together with the branch-and-bound algorithm as proposed in \cite{dey2019new}. 
Specifically, the 
the bilinear constraints (original and aggregated) whose convex hull is used to build the convex relaxation are selected once at the root node. At each node of the tree, we update the bounds on the variables, and recompute the polyhedral outer approximations of the
convex hull of each of these constraints.
The branching rule of the branch-and-bound algorithm, chooses a variable among one of the $x$ variables, such that it minimizes the sum of the volume of resulting new 2-dimensional convex hulls (after fixing variable). This is precisely the same rule as used in \cite{dey2019new}. A termination criterion for the branch-and-bound process is set to be $\rho$ being less than 0.5\%.
\paragraph{Bound reduction: } Before we solve the problem, we reduce the bounds on variables using optimality-based bound tightening (``OBBT'') to reduce the box constraints. These reduced bounds are then used with all the methods for solving the instance. For example, reduced bounds are provided to BARON, McCormick relaxation, one-row relaxation and aggregation relaxation. Details of the bound reduction process are provided in Appendix~\ref{app:bound_reduction}.

\subsubsection{Evaluation at the root node}
We first evaluate the additional root node gap closed in the branch-and-bound tree using the aggregated constraints.
Figure~\ref{fig:fem_rootnode} shows that adding $nm/2$ aggregated constraints can yield an average relative gap improvement of 2.94\%,	2.98\%, and 6.82\% for 12-story data and 2.56\%, 2.84\%, 4.88\% for 16-story data based on weights found by simple search, surrogate search, and grid search respectively. It is clear that the grid search significantly improves the relative root node gap compared to that of the one-row relaxation without aggregated constraints. The simple search and the surrogate search yielded less effective aggregations.

\begin{figure}[tbh!]
  \centering
  \begin{subfigure}[b]{0.48\textwidth}
    \centering
    \includegraphics[width=\textwidth]{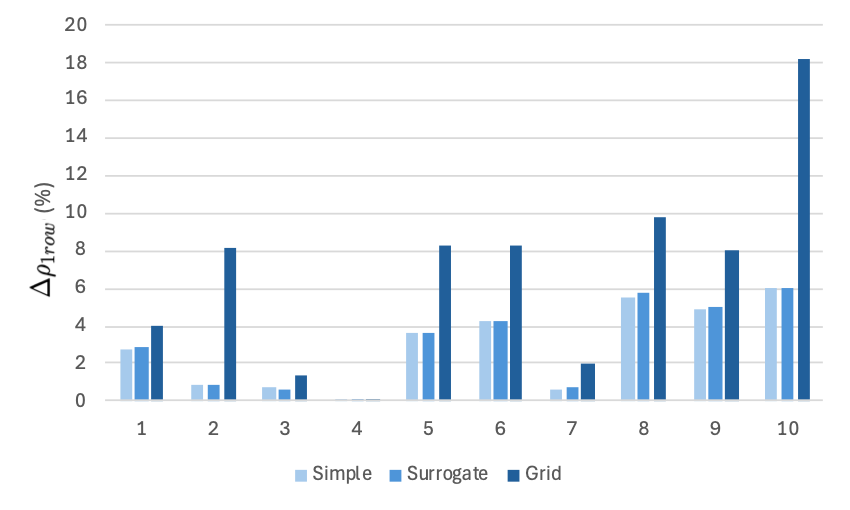}
    \caption{12-story data}
    \label{fig:12story_rootnode}
  \end{subfigure}
  \begin{subfigure}[b]{0.48\textwidth}
    \centering
    \includegraphics[width=\textwidth]{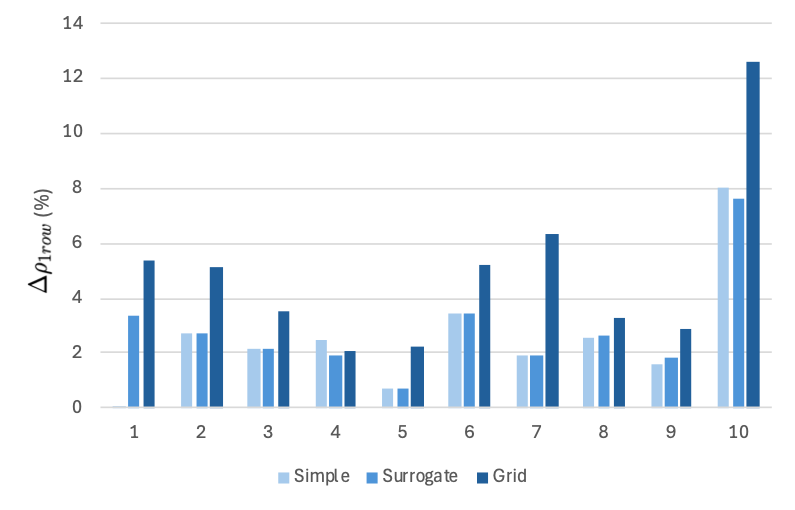}
    \caption{16-story data}
    \label{fig:16story_rootnode}
  \end{subfigure}
  \caption{Relative gap improvement (\%) against the gap achieved from the one-row relaxation for different instances.}
  \label{fig:fem_rootnode}
\end{figure}

Finally, Table~\ref{tab:fem_times} shows the average time spent to find the aggregation weights using different methods. The grid search now requires solving $|G| \times n/2$ optimization problems which is 240 for 12-story and 480 optimization problems. This takes a relatively longer time to solve. However, as we have discussed previously for the case of experiments on random instances, with complete parallelization, the estimated time is comparable to the simple heuristic method.
\begin{table}[tbh!]
\centering
\caption{Average time spent in finding aggregation weights (sec)}
\label{tab:fem_times}
\begin{tabular}{lrrrr}
\toprule
data & \multicolumn{1}{c}{Simple} & \multicolumn{1}{c}{Surrogate} & \multicolumn{1}{c}{Grid} & \multicolumn{1}{c}{Grid (parallel est.)} \\
\midrule
12-story & 2.75 & 28.31 & 169.24 & 4.19 \\
16-story & 6.15 & 82.34 & 1441.51 & 8.36 \\
\bottomrule
\end{tabular}
\end{table}

\subsubsection{Evaluation of bounds obtained by the branch-and-bound tree}
From both the randomized experiment evaluation and the root node evaluation, we see that the grid search provides the most improvement. We therefore compare solving the FEM instances using the global solver BARON, using the one-row relaxation without any aggregated constraints, and using the aggregation relaxation where the aggregation weights are selected through grid search. The time limit is set to 1 hour (3600 seconds) for both the commercial solver and the custom branch-and-bound solver. 

In our preliminary experiments, we found that adding the convex hull of additional $nm/2$ aggregated constraints significantly slows down the solving time at each node of the branch-and-bound tree.
Hence, we limit the number of aggregated constraints to be added to $t$ constraints where $t = 2 $ for the 12-story instances and $t = 5$ for 16-story instances. 
We select the top $t$-aggregated constraints by sorting the distance metrics \eqref{eq:gridsearch} for the $nm/2$ aggregations.

Tables~\ref{tab:12story_final_gap} and ~\ref{tab:16story_final_gap} show the final relative gap at the end of the time limit ($\rho$), the relative gap improvement against using the global solver BARON ($\Delta\rho_{BARON}$), and the number of nodes solved in the branch-and-bound tree. Although we examined 10 instances in each of the 12-story and 16-story data sets, one instance in each data set was solved to optimality by both BARON and the aggregation approach with grid search within $1$ minute, and by one-row relaxation within $10$ minutes-- all well within the time limit; hence it is excluded from the analysis. First, we see that using the one-row relaxation provides a much smaller final relative gap compared to the commercial solver in all instances but instance 8 for the 16-story instances. Secondly, when we compare the one-row relaxation approach and the relaxation with aggregations approach, adding a subset of aggregated constraints excels in closing the final gap in most instances despite fewer nodes being explored.  Furthermore, in 12-story instances, 3 out of 9 instances are solved to optimality by the aggregation technique, whereas only 1 instance is solved to optimality without aggregation.

It is possible that adding aggregated constraints may not benefit the branch-and-bound process. Recent studies~\cite{shah2024non} have shown how adding cuts may increase the size of the branch-and-bound tree due to wrong branching decisions. For instance 3 in 16-story data, even though a similar number of nodes were explored by the time limit, one-row relaxation without aggregated constraints reached a smaller final relative gap.

Nevertheless, adding the aggregated constraints outperforms the one-row relaxation without aggregation in all but one instance in each data set and improves the gap improvement from 68.59\% to 81.55\% for 12-story data and from 24.56\% to 39.22\% for the 16-story data.
\begin{table}[tbh!]
\centering
\caption{Final gap for 12-story data}
\label{tab:12story_final_gap}
\begin{tabular}{lrrrrrrrrr}
\toprule
& \multicolumn{3}{c}{$\rho$ (\%)} & \multicolumn{1}{l}{} & \multicolumn{2}{c}{$\Delta\rho_{BARON}$ (\%)} & \multicolumn{1}{l}{} & \multicolumn{2}{c}{Number of Nodes} \\ 
\cmidrule{2-4} \cmidrule{6-7} \cmidrule{9-10}
Instance & \multicolumn{1}{c}{BARON} & \multicolumn{1}{c}{1Row} & \multicolumn{1}{c}{Grid} & \multicolumn{1}{c}{} & \multicolumn{1}{c}{1Row} & \multicolumn{1}{c}{Grid} & \multicolumn{1}{c}{} & \multicolumn{1}{c}{1Row} & \multicolumn{1}{c}{Grid} \\
\midrule
1 & 3.61 & 0.48 & 0.50$^{\dagger}$ &  & 86.70 & 86.17$^{\dagger}$ &  & 859 & 665 \\
2 & 4.24 & 0.64 & 0.51 &  & 84.94 & 88.04 &  & 1663 & 942 \\
3 & 4.64 & 1.53 & 2.43 &  & 66.92 & 47.63 &  & 1512 & 1080 \\
4 & 4.07 & 0.90 & 0.49 &  & 77.80 & 87.89 &  & 1636 & 700 \\
5 & 3.73 & 2.60 & 0.98 &  & 30.31 & 73.73 &  & 1784 & 937 \\
6 & 4.92 & 0.89 & 0.58 &  & 81.84 & 88.12 &  & 1582 & 1019 \\
7 & 2.54 & 1.16 & 0.33 &  & 54.11 & 87.05 &  & 1680 & 1096 \\
8 & 8.41 & 3.29 & 0.91 &  & 60.85 & 89.21 &  & 1664 & 1083 \\
9 & 4.61 & 1.21 & 0.64 &  & 73.79 & 86.11 &  & 1597 & 1094 \\
\textbf{Average} & \textbf{4.02} & \textbf{1.33} & \textbf{0.75} &  & \textbf{68.59} & \textbf{81.55} & \textbf{} & \textbf{1553} & \textbf{957} \\
\bottomrule
\end{tabular}
\\ \raggedright$^{\dagger}$\small{The termination criterion of relative gap 0.5\% is reached; hence, the branch-and-bound process ended. This instance does not count as an instance adding the aggregation slows down the branch-and-bound process and results in a larger final relative gap.}\\
\end{table}

\begin{table}[tbh!]
\centering
\caption{Final gap for 16-story data}
\label{tab:16story_final_gap}
\begin{tabular}{lrrrrrrrrr}
\toprule
& \multicolumn{3}{c}{$\rho$ (\%)} & \multicolumn{1}{l}{} & \multicolumn{2}{c}{$\Delta\rho_{BARON}$ (\%)} & \multicolumn{1}{l}{} & \multicolumn{2}{c}{Number of Nodes} \\ 
\cmidrule{2-4} \cmidrule{6-7} \cmidrule{9-10}
Instance & \multicolumn{1}{c}{BARON} & \multicolumn{1}{c}{1Row} & \multicolumn{1}{c}{Grid} & \multicolumn{1}{c}{} & \multicolumn{1}{c}{1Row} & \multicolumn{1}{c}{Grid} & \multicolumn{1}{c}{} & \multicolumn{1}{c}{1Row} & \multicolumn{1}{c}{Grid} \\
\midrule
1 & 19.53 & 14.35 & 11.87 &  & 26.55 & 39.24 &  & 496 & 447 \\
2 & 6.44 & 5.67 & 3.73 &  & 11.98 & 42.13 &  & 533 & 368 \\
3 & 5.52 & 3.91 & 4.55 &  & 29.04 & 17.52 &  & 517 & 513 \\
4 & 28.64 & 15.67 & 14.62 &  & 45.29 & 48.95 &  & 682 & 416 \\
5 & 4.57 & 1.72 & 0.67 &  & 62.42 & 85.38 &  & 541 & 384 \\
6 & 9.22 & 5.15 & 3.67 &  & 44.20 & 60.15 &  & 687 & 520 \\
7 & 5.15 & 4.44 & 3.09 &  & 13.90 & 40.00 &  & 518 & 408 \\
8 & 6.14 & 8.57 & 6.98 &  & -39.57 & -13.62 &  & 570 & 531 \\
9 & 6.26 & 4.56 & 4.18 &  & 27.26 & 33.26 &  & 509 & 457 \\
\textbf{Average} & \textbf{10.16} & \textbf{7.11} & \textbf{5.93} &  & \textbf{24.56} & \textbf{39.22} &  & \textbf{561} & \textbf{449} \\
\bottomrule
\end{tabular}
\end{table}

\section{Conclusion }
\label{sec:conclusion}
We studied the theoretical and computational power of aggregating constraints. We showed that the aggregations yield a convex hull of a bounded set with two bilinear bipartite equalities for the two-variable case. As soon as the number of variables increases to 3 or more, a set may need infinite aggregations to yield the convex hull or even infinite aggregations may not yield the convex hull. Although there is a theoretical limitation that the exact convex hull may not be achieved from aggregations, we show computationally that finding ``good'' aggregations using a grid search can provide much tighter convex relaxation of the sets. 
In particular, we applied the aggregation procedure to real-life applications on the FEM update problem, where the aggregations benefit the branch-and-bound process by closing more gaps within a specified time limit. 

\section{Proof of Theorem~\ref{theorem:n1=n2=1}}
\label{sec:proof_n1=n2=1}

\begin{observation} 
\label{observation:hyperbola_geometry}
We use the following elementary observations regarding the geometry of the 2-dimensional hyperbola several times in the proof.  Given a  hyperbola:
$$G:= \{(x,y)\in \mathbb{R}^2\,|\,  (x - u)(y - v) = w\}, $$
\begin{itemize}
\item If $w = 0$, $G$ is just the union of the two lines $\{(x,y)\,|\, x = u\}$ and $\{(x,y)\,|\, y = v\}$.
\item If $w>0$, then $
G$ is the union of two connected components, one component is contained in $\{(x,y)\,|\, x > u, y > v\}$ and the other component is contained in $\{ (x,y)\,|\, x < u, y < v\}.$ Moreover, $y  = \frac{w}{x - u} + v$ (equation describing $G$) is convex   in $\{(x,y)\,|\, x > u, y > v\}$ and concave in $\{ (x,y)\,|\, x < u, y < v\}.$  In both branches, $y  = \frac{w}{x - u} + v$  is monotonically decreasing. 

\item If $w < 0$, then $G$ is the union of two connected components (also referred to as branches), one contained in $\{(x,y)\,|\, x < u, y > v\}$ and the other contained in $\{(x,y)\,|\, x > u, y < v\}$. Moreover $y = \frac{w}{x - u} + v$ (equation describing $G$) is convex in $\{(x,y)\,|\, x < u, y > v\}$ and concave in $\{(x,y)\,|\, x > u, y < v\}$. In both branches, $y = \frac{w}{x - u} + v$  is monotonically increasing.
\item In particular, if $u \notin (0,1)$ or $v\notin (0,1)$, then only one branch of $G$ intersects the $[0,\ 1]^2$ box.
\end{itemize}
\end{observation}
For the purpose of this proof, the set $S$ may be written as:
\begin{eqnarray}
\label{eq:S_original}
S = \left\{ x,y \in [0,1]^2 \, \left | \,  
\begin{array}{lcl} 
xy + a_1 x + b_1 y + c_1 & = & 0 \\ 
xy + a_2 x + b_2 y + c_2 & = & 0
\end{array} \right. \right\},
\end{eqnarray}
since if $q_1 = q_2 =0$, then there is nothing to prove and if one of $q_1, q_2$ is zero, then we may take a suitable aggregation to arrive at (\ref{eq:S_original}).

The general scheme of the proof is illustrated in Figure~\ref{fig:n1=n2=1_proof_sketch}. Figure~\ref{fig:n1=n2=1_proof_sketch_a} shows the convex hull of each constraint separately, and their intersection significantly overestimates the convex hull of $S$. In Figure~\ref{fig:n1=n2=1_proof_sketch_b}, by taking an affine transformation of the constraints to make one of the constraints linear, we can represent $S$ equivalently in the following form:
\begin{eqnarray}
\label{eq:S_equiv}
S = \left\{ x,y \in [0,1]^2 \, \left | \,  
\begin{array}{lcl} 
xy + a_1 x + b_1 y + c_1 & = & 0 \\ 
m x - y + b & = & 0
\end{array} \right. \right\},
\end{eqnarray}
where $m = -(a_1 - a_2)/(b_1 - b_2)$ and $b = -(c_1 - c_2)/(b_1 - b_2)$.
Note that equivalently the second constraint in \eqref{eq:S_equiv} is an aggregation of the two constraints in \eqref{eq:S_original}.  
Furthermore, from Assumption~\ref{assumption:1} and Assumption~\ref{assumption:2}, in case $b_1 - b_2 = 0$, it would imply that $a_1 - a_2 \neq 0$. Thus, if $b_1 - b_2 = 0$, then we proceed with the rest of the proof by writing the linear constraint in (\ref{eq:S_equiv}) the form of $my - x + b = 0$ with $m = (b_1 - b_2)/(a_1 - a_2)$ and $b = - (c_1 - c_2) / (a_1 - a_2)$ and treat $x$ as $y$ and $y$ as $x$ in the following proofs.

In Figure~\ref{fig:n1=n2=1_proof_sketch_c}, we make one more linear transformation and arrive at the following form of $S$:
\begin{eqnarray}
\label{eq:S_equiv_2}
S = \left\{ x,y \in [0,1]^2 \, \left | \,  
\begin{array}{lcl} 
(x-r)y - \tau & = & 0 \\ 
m x - y + b & = & 0
\end{array} \right. \right\}
\end{eqnarray}
where $r = - \left(b_1 + \frac{a_1}{m} \right)$ and $\tau = - \left( c_1 - \frac{a_1 b}{m} \right)$. For notational simplicity, define:
\begin{align*}
    & S_1 = [0,1]^2 \cap H_1, \textup{ where } H_1 = 
\{x, y \in \mathbb{R}^2 \ | \ (x-r)y=\tau\} \\
    & S_2 = [0,1]^2 \cap l, \textup{ where } l = 
    \{(x,y) \in \mathbb{R}^2 \,|\, mx-y+b=0\}.
\end{align*}

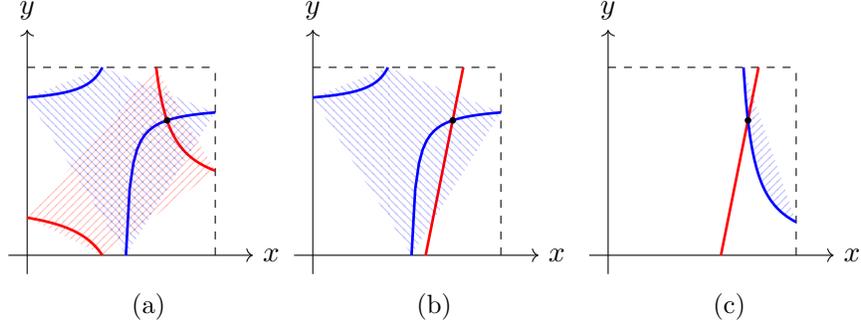
\begin{figure}[tbh!]
\centering
\begin{minipage}{0.23\textwidth}
    \centering
    \begin{tikzpicture}[scale=2.5]
    \draw[dashed] (1,0) -- (1,1);
    \draw[dashed] (0,1) -- (1,1);
    \draw[->] (-0.1,0) -- (1.2,0) node[right] {$x$};
    \draw[->] (0,-0.1) -- (0,1.2) node[above] {$y$};
    \draw[red,domain=0:0.4,line width=1pt] plot (\x,{0.06/(\x-0.6)+0.3});
    \draw[red,domain=0.06/(1-0.3)+0.6:1,line width=1pt] plot (\x,{0.06/(\x-0.6)+0.3});
    \draw[blue,domain=0:0.5-0.02/(1-0.8), line width=1pt] plot (\x,{0.8-0.02/(\x-0.5)});
    \draw[blue,domain=0.5-0.02/(0-0.8):1, line width=1pt] plot (\x,{0.8-0.02/(\x-0.5)});
    \fill (0.744,0.718) circle[radius=0.5pt];
    \fill[pattern=north east lines,pattern color=red, opacity=0.5] (0,{0.06/(0-0.6)+0.3}) -- ({0.06/(1-0.3)+0.6},1) -- (1,{0.06/(1-0.6)+0.3}) -- ({0.06/(0-0.3)+0.6},0) -- cycle;
    \fill[pattern=north west lines,pattern color=blue, opacity=0.5] (0,{0.8-0.02/(-0.5)}) -- ({0.5-0.02/(1-0.8)},1) -- (1,{0.8-0.02/(1-0.5)}) -- ({0.5-0.02/(0-0.8)},0) -- cycle;
    \end{tikzpicture}
    \subcaption{}
    \label{fig:n1=n2=1_proof_sketch_a}
\end{minipage}%
\begin{minipage}{0.23\textwidth}
    \centering
    \begin{tikzpicture}[scale=2.5]
    \draw[dashed] (1,0) -- (1,1);
    \draw[dashed] (0,1) -- (1,1);
    \draw[->] (-0.1,0) -- (1.2,0) node[right] {$x$};
    \draw[->] (0,-0.1) -- (0,1.2) node[above] {$y$};
    \draw[blue,domain=0:0.5-0.02/(1-0.8), line width=1pt] plot (\x,{0.8-0.02/(\x-0.5)});
    \draw[blue,domain=0.5-0.02/(0-0.8):1, line width=1pt] plot (\x,{0.8-0.02/(\x-0.5)});
    \draw[red,domain=0.6:0.8,line width=1pt] plot (\x,{5*\x-3});
    \fill (0.744,0.718) circle[radius=0.5pt];
    \fill[pattern=north west lines,pattern color=blue, opacity=0.5] (0,{0.8-0.02/(-0.5)}) -- ({0.5-0.02/(1-0.8)},1) -- (1,{0.8-0.02/(1-0.5)}) -- ({0.5-0.02/(0-0.8)},0) -- cycle;
    \end{tikzpicture}
    \subcaption{}
    \label{fig:n1=n2=1_proof_sketch_b}
\end{minipage}
\begin{minipage}{0.23\textwidth}
    \centering
    \begin{tikzpicture}[scale=2.5]
    \draw[dashed] (1,0) -- (1,1);
    \draw[dashed] (0,1) -- (1,1);
    \draw[->] (-0.1,0) -- (1.2,0) node[right] {$x$};
    \draw[->] (0,-0.1) -- (0,1.2) node[above] {$y$};
    \draw[red,domain=0.6:0.8,line width=1pt] plot (\x,{5*\x-3});
    \draw[blue,domain=0.72:1, line width=1pt] plot (\x,{0.06/(\x-0.66)});
    \fill (0.744,0.718) circle[radius=0.5pt];
    \begin{scope}
      \fill[pattern=north east lines,pattern color=blue, opacity=0.5, domain=0.72:1, variable=\x]
        (0.72,1) plot (\x,{0.06/(\x-0.66)}) -- plot (\x,{-50/17*\x+53/17}) -- (1,6/34) -- cycle;
    \end{scope}
    \end{tikzpicture}
    \subcaption{}
    \label{fig:n1=n2=1_proof_sketch_c}
\end{minipage}
\caption{Shaded areas in (a) represent the convex hull of each constraint in their original form with the true feasible region shown as a black dot.}
\label{fig:n1=n2=1_proof_sketch}
\end{figure}

We discuss the possible ``shapes" of the set $S$:
\begin{enumerate}
\item $\tau \neq 0$: In this case, $S_1$ intersects with $[0,1]^2$ box on only one branch of the hyperbola -- we call this branch the feasible branch. There are two subcases based on how $S_2$ intersects with $S_1$:
\begin{enumerate}
\item \label{item:tau_notzero_line} $S_2$ intersects $S_1$ at two points in $[0,1]^2$ box. In this case, $\textup{conv}(S)$ is a line segment.
\item \label{item:tau_notzero_point} $S_2$ intersects $S_1$ at one point in $[0,1]^2$ box. In this case, $\textup{conv}(S)$ is a point.
\end{enumerate}
\item $\tau = 0$: In this case, $S_1$ is the product of two lines intersected with the $[0,1]^2$ box. Further, if $r \not\in [0, \ 1]$, then $S_1$ is a line segment and there is nothing to prove. 
In case $r \in [0, \ 1]$, if $S_2$ is the line $y = 0$ or $x = r$, then there there is nothing to prove. Otherwise, similar to the previous case, there are two possibilities:
\begin{enumerate}
\item \label{item:tau_zero_line} $S_2$ intersects $S_1$ at two points in $[0,1]^2$ box. In this case, $\textup{conv}(S)$ is a line segment.
\item \label{item:tau_zero_point} $S_2$ intersects $S_1$ at one point in $[0,1]^2$ box. In this case, $\textup{conv}(S)$ is a point.
\end{enumerate}
\end{enumerate}

Figure~\ref{fig:n1=n2=1_S_shapes} shows different shapes of $S$ which are intersections of the red curve and blue line. There are at most 2 elements in $S$ and are marked by black dots. The respective convex hulls of $S_1$ are shown in red shaded area. When $\conv(S) \neq \conv(S_1) \cap S_2$, we find another aggregation $S_3$ that is represented in green with its respective convex hulls represented in green shaded area. Now we prove for different shapes of $S$ how we obtain the $\conv(S)$.
\begin{figure}[tbh!]
\centering
\begin{minipage}{0.23\textwidth}
    \centering
    \begin{tikzpicture}[scale=2.5]
    \draw[dashed] (1,0) -- (1,1);
    \draw[dashed] (0,1) -- (1,1);
    \draw[->] (-0.1,0) -- (1.2,0) node[right] {$x$};
    \draw[->] (0,-0.1) -- (0,1.2) node[above] {$y$};
    \draw[red,domain=0.1:1,line width=1pt] plot (\x,{0.2/(\x+0.1)});
    \draw[blue,domain=0.0:1,line width=1pt] plot (\x,{(-10*\x+12)/15});
    \fill (0.2,2/3) circle[radius=0.8pt];
    \fill (0.9,0.2) circle[radius=0.8pt];
    \draw[domain=0.2:0.9,line width=1.8pt] plot (\x,{-2/3*\x+4/5});
    \begin{scope}
      \fill[pattern=north east lines,pattern color=red, opacity=0.5, domain=0.1:1, variable=\x]
        (0.1,1) plot (\x,{0.2/(\x+0.1)}) -- plot (\x,{-10/11*\x+12/11}) -- (1,2/11) -- cycle;
    \end{scope}
    \end{tikzpicture}
    \subcaption*{$\conv(S)$ for \ref{item:tau_notzero_line} case}
    \label{fig:n1=n2=1_shape_conv_1a}
\end{minipage}%
\begin{minipage}{0.23\textwidth}
    \centering
    \begin{tikzpicture}[scale=2.5]
    \draw[dashed] (1,0) -- (1,1);
    \draw[dashed] (0,1) -- (1,1);
    \draw[->] (-0.1,0) -- (1.2,0) node[right] {$x$};
    \draw[->] (0,-0.1) -- (0,1.2) node[above] {$y$};
    \draw[red,domain=0.1:1,line width=1pt] plot (\x,{0.2/(\x+0.1)});
    \draw[blue,domain=0.0:1,line width=1pt] plot (\x,{(2*\x+1)/4});
    \draw[darkgreen,domain=0.0:0.6,line width=1pt] plot (\x,{-5/8/(\x-1) - 0.55});
    \fill (0.36332, 0.43166) circle[radius=0.8pt];
    \begin{scope}
      \fill[pattern=north east lines,pattern color=red, opacity=0.5, domain=0.1:1, variable=\x]
        (0.1,1) plot (\x,{0.2/(\x+0.1)}) -- plot (\x,{-10/11*\x+12/11}) -- (1,2/11) -- cycle;
    \end{scope}
    \begin{scope}
      \fill[pattern=north west lines,pattern color=darkgreen, opacity=0.5, domain=0:0.6, variable=\x]
        (0,3/40) plot (\x,{-5/8/(\x-1) - 0.55}) -- plot (\x,{37/24*\x+3/40}) -- (0.6,1) -- cycle;
    \end{scope}
    \end{tikzpicture}
    \subcaption*{$\conv(S)$ for \ref{item:tau_notzero_point} case}
    \label{fig:n1=n2=1_shape_conv_1b}
\end{minipage}%
\begin{minipage}{0.23\textwidth}
    \centering
    \begin{tikzpicture}[scale=2.5]
    \draw[dashed] (1,0) -- (1,1);
    \draw[dashed] (0,1) -- (1,1);
    \draw[->] (-0.1,0) -- (1.2,0) node[right] {$x$};
    \draw[->] (0,-0.1) -- (0,1.2) node[above] {$y$};
    \draw[red,line width=1pt] (0,0) -- (1,0);
    \draw[red,line width=1pt] (0.3,0) -- (0.3,1);
    \draw[blue,domain=0:0.8,line width=1pt] plot (\x,{-\x+0.8});
    \draw[darkgreen,domain=24/130:0.8, line width=1.5pt] plot (\x,{-0.3+0.24/(\x)});
    \fill (0.3, 0.5) circle[radius=0.8pt];
    \fill (0.8, 0) circle[radius=0.8pt];
    \draw[domain=0.3:0.8,line width=1.8pt] plot (\x,{-\x+0.8});
    \fill[pattern=north east lines,pattern color=red, opacity=0.5] (0,0) -- (0.3,1) -- (1,0) -- cycle;
    \begin{scope}
      \fill[pattern=north west lines,pattern color=darkgreen, opacity=0.5, domain=24/130:0.8, variable=\x]
        (24/130,1) plot (\x,{-0.3+0.24/(\x)}) -- plot (\x,{-13/8*\x + 13/10}) -- (0.8,0) -- cycle;
    \end{scope}
    \end{tikzpicture}
    \subcaption*{$\conv(S)$ for \ref{item:tau_zero_line} case}
    \label{fig:n1=n2=1_shape_conv_2a}
\end{minipage}%
\begin{minipage}{0.23\textwidth}
    \centering
    \begin{tikzpicture}[scale=2.5]
    \draw[dashed] (1,0) -- (1,1);
    \draw[dashed] (0,1) -- (1,1);
    \draw[->] (-0.1,0) -- (1.2,0) node[right] {$x$};
    \draw[->] (0,-0.1) -- (0,1.2) node[above] {$y$};
    \draw[red,line width=1pt] (0,0) -- (1,0);
    \draw[red,line width=1pt] (0.3,0) -- (0.3,1);
    \draw[blue,domain=0:1,line width=1pt] plot (\x,{(\x+0.4)/2});
    \draw[darkgreen,domain=0:86/135, line width=1pt] plot (\x,{-0.35-0.49/(\x-1)});
    \fill (0.3, 0.35) circle[radius=0.8pt];
    \begin{scope}
      \fill[pattern=north west lines,pattern color=darkgreen, opacity=0.5, domain=0:86/135, variable=\x]
        (0,0.14) plot (\x,{-0.35-0.49/(\x-1)}) -- plot (\x,{1.35*\x + 0.14}) -- (86/135,1) -- cycle;
    \end{scope}
    \fill[pattern=north east lines,pattern color=red, opacity=0.5] (0,0) -- (0.3,1) -- (1,0) -- cycle;
    \draw[darkgreen,domain=6/85:1, line width=1pt] plot (\x,{0.15+0.06/(\x)});
    \fill (0.3, 0.35) circle[radius=0.8pt];
    \begin{scope}
      \fill[pattern=north west lines,pattern color=darkgreen, opacity=0.5, domain=6/85:1, variable=\x]
        (6/85,1) plot (\x,{0.15+0.06/(\x)}) -- plot (\x,{-0.85*\x + 1.06}) -- (1,0.21) -- cycle;
    \end{scope}
    \end{tikzpicture}
    \subcaption*{$\conv(S)$ for \ref{item:tau_zero_point} case}
    \label{fig:n1=n2=1_shape_conv_2b}
\end{minipage}%
\caption{Different shapes of the set $S$ and their respective convex hull.}
\label{fig:n1=n2=1_S_shapes}
\end{figure}
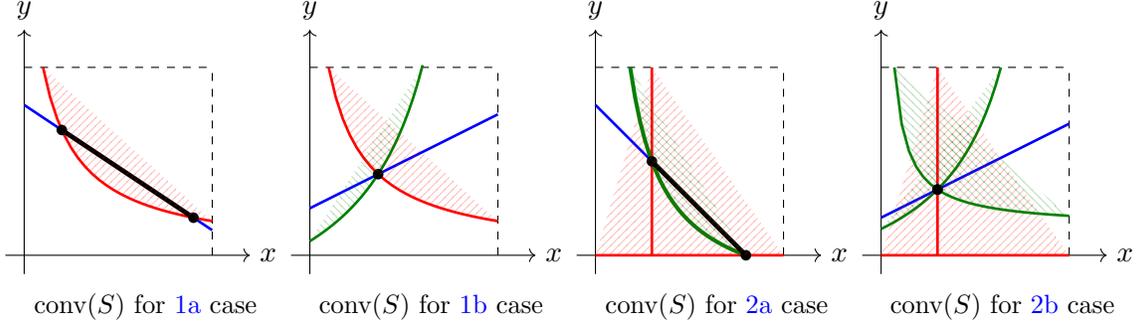

\paragraph{Proof of \ref{item:tau_notzero_line}}
If a hyperbola $S_1$ with $\tau \neq 0$ and a line $S_2$ intersect at two points and both points are on the same branch of the hyperbola $S_1$, then $\conv(S_1) \cap \conv(S_2)$ is the line segment connecting the two points belonging to $S_1 \cap S_2$ (see Figure~\ref{fig:n1=n2=1_S_shapes}). Then $\conv(S)=\conv(S_1) \cap \conv(S_2)$.
\paragraph{Proof of \ref{item:tau_notzero_point}}

A special case of \ref{item:tau_notzero_point} is when the line describing $S_2$ is a tangent to the feasible branch of the hyperbola corresponding to $S_1$.
In this case, the proof is the same as the proof for \ref{item:tau_notzero_line}.

In all other cases, the idea is that we can find at least one aggregation that will have one branch of the hyperbola intersecting with $[0,1]^2$ box, and the convex hull of this hyperbola intersects with the line segment $mx - y + b=0$ on the ``opposite side" of the line intersecting the convex hull of $S_1$. 

Let $(x^*, y^*)$ be the singleton in the set $S$. When $\tau > 0$, there are three cases for us to consider: 
\begin{enumerate}
    \item $m \geq 0$;
    \item $m < 0$ and $1-b \geq m (r + \tau)$; and
    \item $m < 0$ and $1-b \leq m (r + \tau)$.
\end{enumerate}
The intersection between $\conv(S_1)$ and the line $S_2$ is shown as a black line segment in Figure~\ref{fig:tau>0_convrightside}). 
\begin{figure}[hbt!]
\centering
\begin{minipage}{0.32\textwidth}
    \centering
    \begin{tikzpicture}[scale=2.5]
    \draw[dashed] (1,0) -- (1,1);
    \draw[dashed] (0,1) -- (1,1);
    \draw[->] (-0.1,0) -- (1.2,0) node[right] {$x$};
    \draw[->] (0,-0.1) -- (0,1.2) node[above] {$y$};
    \draw[red,domain=0.1:1,line width=1pt] plot (\x,{0.2/(\x+0.1)});
    \draw[blue,domain=0.0:1,line width=1pt] plot (\x,{(2*\x+1)/4});
    \draw[black,domain=0.36332:0.6,line width=2pt] plot (\x,{(2*\x+1)/4});
    \draw[dashed] (0.36332,0.43166) -- (0.36332,0) node[below] {$x^*$};
    \fill (0.36332, 0.43166) circle[radius=0.8pt];
    \begin{scope}
      \fill[pattern=north east lines,pattern color=red, opacity=0.5, domain=0.1:1, variable=\x]
        (0.1,1) plot (\x,{0.2/(\x+0.1)}) -- plot (\x,{-10/11*\x+12/11}) -- (1,2/11) -- cycle;
    \end{scope}
    \end{tikzpicture}
    \subcaption{\footnotesize $m \geq 0$}
    \label{fig:n1=n2=1_tau_nonzero_point_case1}
\end{minipage}%
\begin{minipage}{0.32\textwidth}
    \centering
    \begin{tikzpicture}[scale=2.5]
    \draw[dashed] (1,0) -- (1,1);
    \draw[dashed] (0,1) -- (1,1);
    \draw[->] (-0.1,0) -- (1.2,0) node[right] {$x$};
    \draw[->] (0,-0.1) -- (0,1.2) node[above] {$y$};
    \draw[red,domain=0.1:1,line width=1pt] plot (\x,{0.2/(\x+0.1)});
    \draw[blue,domain=0.0:1,line width=1pt] plot (\x,{(-5*\x+8)/10});
    \draw[dashed] (0.18211, 0.70895) -- (0.18211, 0) node[below] {$x^*$};
    \fill (0.18211, 0.70895) circle[radius=0.8pt];
    \draw[domain=0.18211:0.72,line width=2pt] plot (\x,{(-5*\x+8)/10});
    \begin{scope}
      \fill[pattern=north east lines,pattern color=red, opacity=0.5, domain=0.1:1, variable=\x]
        (0.1,1) plot (\x,{0.2/(\x+0.1)}) -- plot (\x,{-10/11*\x+12/11}) -- (1,2/11) -- cycle;
    \end{scope}
    \end{tikzpicture}
    \subcaption{\footnotesize $m < 0$, $(m+b)(1-r) \geq \tau$}
    \label{fig:n1=n2=1_tau_nonzero_point_case2}
\end{minipage}%
\begin{minipage}{0.32\textwidth}
    \centering
    \begin{tikzpicture}[scale=2.5]
    \draw[dashed] (1,0) -- (1,1);
    \draw[dashed] (0,1) -- (1,1);
    \draw[->] (-0.1,0) -- (1.2,0) node[right] {$x$};
    \draw[->] (0,-0.1) -- (0,1.2) node[above] {$y$};
    \draw[red,domain=0.1:1,line width=1pt] plot (\x,{0.2/(\x+0.1)});
    \draw[blue,domain=0.5:1,line width=1pt] plot (\x,{(-10*\x+10)/5});
    \draw[dashed] (0.9, 0.2) -- (0.9, 0) node[below] {$x^*$};
    \fill (0.9, 0.2) circle[radius=0.8pt];
    \draw[domain=0.83:0.9,line width=2pt] plot (\x,{(-10*\x+10)/5});
    \begin{scope}
      \fill[pattern=north east lines,pattern color=red, opacity=0.5, domain=0.1:1, variable=\x]
        (0.1,1) plot (\x,{0.2/(\x+0.1)}) -- plot (\x,{-10/11*\x+12/11}) -- (1,2/11) -- cycle;
    \end{scope}
    \end{tikzpicture}
    \subcaption{\footnotesize $m < 0$, $(m+b)(1-r) \leq \tau$}
    \label{fig:n1=n2=1_tau_nonzero_point_case3}
\end{minipage}%
\caption{Red shaded area represents $\conv(S_1)$ and the blue line is $\conv(S_2) = S_2$. The black point represents $\conv(S)$ and the black line segment represents $\conv(S_1) \cap S_2$.}
\label{fig:tau>0_convrightside}
\end{figure}
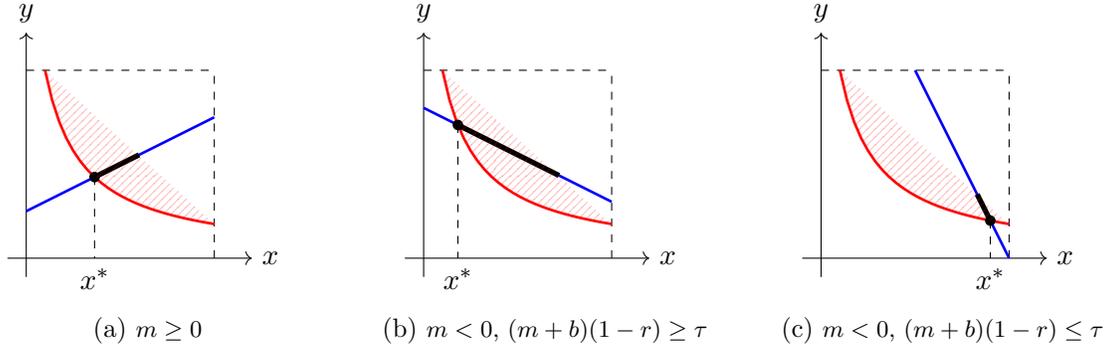

\begin{enumerate}
\item $m \geq 0$: We can make the following aggregation from \eqref{eq:S_equiv_2} by multiplying $(1 + \epsilon -r)$ to the second equation and adding it to the first equation:
\begin{eqnarray*}
\begin{array}{rrlcl} 
    & & (x- r) y - \tau & = & 0 \\ 
    + & (1+ \epsilon - r) \quad \times & mx - y + b & = & 0 \\
    = & & \left( x - \tilde r \right) \left( y - \tilde s \right) - \tilde \tau & = & 0
\end{array}
\end{eqnarray*}
where $\tilde r = 1 + \epsilon$, $\tilde s = (r-1-\epsilon)m$ and $\tilde \tau = \tau + (r-1-\epsilon)(b+m+m\epsilon)$. The equation $\tilde \tau = 0$ is a quadratic function of $\epsilon$; hence, we can always find $\epsilon \geq 0$ such that $\tilde \tau \neq 0$.
Furthermore, note that for the line $l$ and $H_1$ to have an intersection in $[0,1]^2$ box given that $m \geq 0$ and $\tau > 0$, it must be that the $y$-coordinate of $l$ is greater than or equal to the $y$-coordinate of $H_1$ at $x=1$:
$$m + b \geq \frac{\tau}{1-r}.$$
This implies that $(m+b)(1-r) \geq \tau$ as $r<1$ for $H_1$ to have an intersection in $[0,1]^2$ box. Then, we have: 
\begin{align*}
    \tilde \tau & = \tau - (1+\epsilon - r) (m + m\epsilon +b) \leq \tau - (1 - r) (m + b) \leq 0.
\end{align*}
Since $\tilde \tau \neq 0$, we obtain $\tilde \tau < 0.$

We define the following set:
$$S_3 = [0,\ 1]^2 \cap H_3, \textup{ where } H_3 = 
\{x, y \in \mathbb{R}^2 \ | \ (x-\tilde r)(y-\tilde s) = \tilde \tau \}.$$ 
Since $S_3$ is defined by the branch of $H_3$ satisfying $x < \tilde r$, henceforth, we let $H_3 = 
\{x, y \in \mathbb{R}^2 \ | \ (x-\tilde r)(y-\tilde s) = \tilde \tau, x < \tilde r\}$.

We note that $\conv(S_1) \cap S_2$ is on the line $l$ along the direction that $x \geq x^*$ as illustrated in Figure~\ref{fig:n1=n2=1_tau_nonzero_point_case1}. 
We claim that $\conv(S_3) \cap S_2$ is on the line $l$ along the direction that $x \leq x^*$ so that $\conv(S_1) \cap S_2 \cap \conv(S_3)$ is the singleton $(x^*,y^*)$.
Equivalently, we would like to show that the half-line:
\begin{eqnarray}
\label{eq:h3_l_intersect_case1}
\{(x,y) \in l\,|\, x > x^* \} \cap \textup{conv}(H_3) = \emptyset.
\end{eqnarray} 
Note that following Observation~\ref{observation:hyperbola_geometry}, only one branch of $H_3$ intersects $[0,1]^2$ box. 
Since we have that $\tilde \tau < 0$, \eqref{eq:h3_l_intersect_case1} is implied by: 
\begin{eqnarray}
\label{eq:min_h3_line}
mx + b < \textup{min}\{ y\,|\, (x,y) \in \textup{conv}(H_3)\} = \tilde s + \frac{\tilde \tau}{x-\tilde r} \quad \forall \ x > x^*.
\end{eqnarray}
We first show that \eqref{eq:min_h3_line} holds for some $x \in (x^*, 1)$. Since $S = S_1 \cap S_2 = S_3 \cap S_2$, there is no other $(x,y) \neq (x^*,y^*)$ in the $[0,1]^2$ box such that $(x,y) \in H_3 \cap l$. Furthremore, $\tilde s + \frac{\tilde \tau}{x - \tilde r}$ is convex and monotonically increasing, so if we find $x \in (x^*,1)$ such that \eqref{eq:min_h3_line} holds, \eqref{eq:min_h3_line} holds for all $x > x^*$. Otherwise, by the Mean Value Theorem, $H_3$ and $l$ should intersect again in $[0,1]^2$ box or \eqref{eq:min_h3_line} does not hold for all $x \in (x^*,1)$.

Let $\tilde x$ be the $x$-coordinate of $H_3$ at $y=1$. We want to evaluate the $y$-coordinate of $l$ at $\tilde x$ and check if it is less than 1. Let this $y$-coordinate be $\tilde y$.
\begin{align*}
    \tilde x & = \tilde r + \frac{\tilde \tau}{1-\tilde s} 
    = 1 + \epsilon + \frac{\tau + (r-1-\epsilon)(b+m+m\epsilon)}{1+(1+\epsilon-r)m} \\
    \tilde y & = m \tilde x + b \\
    & = m\left(1 + \epsilon + \frac{\tau + (r-1-\epsilon)(b+m+m\epsilon)}{1 +(1+\epsilon-r)m} \right) + b \\
    & = \frac{m}{1 +(1+\epsilon-r)m} ( 1 + \epsilon +(1+\epsilon-r)(1+\epsilon)m + \tau + (r-1-\epsilon)(b+m+m\epsilon)) + b \\
    & = \frac{m}{1 +(1+\epsilon-r)m} ( 1 + \epsilon + \tau + (r - 1-\epsilon)b) + b \\
    & = \frac{m}{1 +(1+\epsilon-r)m} ( 1 + \epsilon + \tau + (r - 1-\epsilon)b) + \frac{b(1 + (1+\epsilon-r)m)}{1+(1+\epsilon-r)m} \\
    & = \frac{b + (1 + \epsilon + \tau)m}{1 + (1+\epsilon-r)m} \\
    & \leq \frac{1 + (1 + \epsilon - r)m}{1 + (1+\epsilon-r)m} = 1
\end{align*}
where the last inequality follows from the fact that $m(\tau + r) \leq 1-b$. When $m=0$, the line equation is $y=b$ and it can only have an intersection with $[0,1]^2$ box if $b \in [0,1]$ so $1-b \geq 0$ and the inequality is satisfied. When $m \neq 0$, $\tau + r$ is the $x$-coordinate of $H_1$ at $y=1$ which is less than or equal to the $x$-coordinate of $l$ at $y=1$ (i.e., $\frac{1-b}{m}$).

\begin{figure}[hbt!]
\centering
\begin{minipage}{0.32\textwidth}
    \centering
    \begin{tikzpicture}[scale=2.5]
    \draw[dashed] (1,0) -- (1,1);
    \draw[dashed] (0,1) -- (1,1);
    \draw[->] (-0.1,0) -- (1.2,0) node[right] {$x$};
    \draw[->] (0,-0.1) -- (0,1.2) node[above] {$y$};
    \draw[red,domain=0.1:1,line width=1pt] plot (\x,{0.2/(\x+0.1)});
    \draw[blue,domain=0.0:1,line width=1pt] plot (\x,{(2*\x+1)/4});
    \draw[black,domain=0.36332:1,line width=2pt] plot (\x,{(2*\x+1)/4});
    \draw[dashed] (0.36332,0.43166) -- (0.36332,0) node[below] {$x^*$};
    \fill (0.36332, 0.43166) circle[radius=0.8pt];
    \begin{scope}
      \fill[pattern=north east lines,pattern color=red, opacity=0.5, domain=0.1:1, variable=\x]
        (0.1,1) plot (\x,{0.2/(\x+0.1)}) -- (1,2/11) -- (1,1) -- cycle;
    \end{scope}
    \end{tikzpicture}
    \subcaption{\footnotesize $\conv(H_1)$ and $S_2$}
\end{minipage}%
\begin{minipage}{0.32\textwidth}
    \centering
    \begin{tikzpicture}[scale=2.5]
    \draw[dashed] (1,0) -- (1,1);
    \draw[dashed] (0,1) -- (1,1);
    \draw[->] (-0.1,0) -- (1.2,0) node[right] {$x$};
    \draw[->] (0,-0.1) -- (0,1.2) node[above] {$y$};
    \draw[blue,domain=0.0:1,line width=1pt] plot (\x,{(2*\x+1)/4});
    \draw[darkgreen,domain=0.0:0.6,line width=1pt] plot (\x,{-5/8/(\x-1) - 0.55});
    \draw[black,domain=0:0.36332,line width=2pt] plot (\x,{(2*\x+1)/4});
    \draw[dashed] (0.36332,0.43166) -- (0.36332,0) node[below] {$x^*$};
    \fill (0.36332, 0.43166) circle[radius=0.8pt];
    \begin{scope}
      \fill[pattern=north west lines,pattern color=darkgreen, opacity=0.5, domain=0:0.6, variable=\x]
        (0,3/40) plot (\x,{-5/8/(\x-1) - 0.55}) -- (0.6,1) -- (0,1) -- cycle;
    \end{scope}
    \end{tikzpicture}
    \subcaption{\footnotesize $\conv(H_3)$ and $S_2$}
\end{minipage}%
\caption{When $\tau > 0$ and $m > 0$, the branch of hyperbola $H_1$ intersecting with $[0,1]^2$ box is convex and monotonically decreasing whereas the branch of hyperbola $H_3$ intersecting with $[0,1]^2$ box is convex and monotonically increasing. Hence, $\min \{ y \ | \ (x,y) \in \conv(H_3) \}$ is precisely the curve represented by $H_3$ which always lies above the line $l$ for $x > x^*$. }
\label{fig:tau>0_m>0_case}
\end{figure}
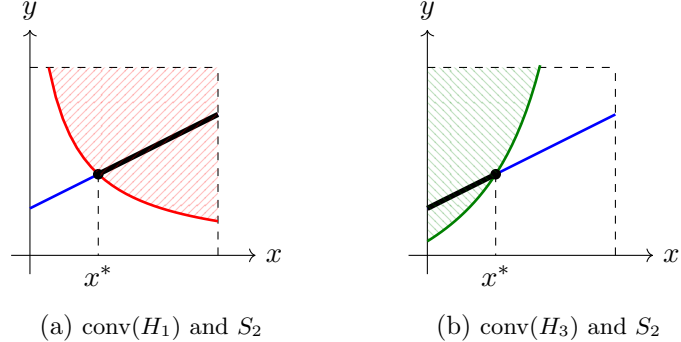

\item $m < 0$ and $(m+b)(1-r) \geq \tau$: We make the same aggregation as above and define $S_3$ and $H_3$ in the same way.
We note that:
\begin{align*}
    \tilde \tau = \tau - (1+\epsilon - r) (m + m\epsilon + b) \leq \tau - (1-r)(m+b) \leq 0.
\end{align*}
We claim that $\conv(S_3) \cap S_2$ is on the line $l$ along the direction that $x \leq x^*$ so that $\conv(S_1) \cap S_2 \cap \conv(S_3)$ is the singleton $(x^*,y^*)$.
Equivalently, we would like to show that the half-line:
\begin{eqnarray}
\label{eq:h3_l_intersect_case2}
\{(x,y) \in l\,|\, x > x^* \} \cap \textup{conv}(H_3) = \emptyset.
\end{eqnarray} 
Note that following Observation~\ref{observation:hyperbola_geometry}, only one branch of $H_3$ intersects $[0,1]^2$ box. 
Since we conclude that $\tilde \tau < 0$, \eqref{eq:h3_l_intersect_case2} implied by: 
\begin{eqnarray}
\label{eq:min_h3_line_case2}
mx + b < \textup{min}\{ y\,|\, (x,y) \in \textup{conv}(H_3)\} = \tilde s + \frac{\tilde \tau}{x-\tilde r} \quad \forall \ x > x^*.  
\end{eqnarray}
which holds because $m < 0$ and $\tilde s + \frac{\tilde \tau}{x-\tilde r}$ is monotonically increasing. Figure~\ref{fig:tau>0_m<0_case1} illustrates the proof.

\begin{figure}[hbt!]
\centering
\begin{minipage}{0.32\textwidth}
    \centering
    \begin{tikzpicture}[scale=2.5]
    \draw[dashed] (1,0) -- (1,1);
    \draw[dashed] (0,1) -- (1,1);
    \draw[->] (-0.1,0) -- (1.2,0) node[right] {$x$};
    \draw[->] (0,-0.1) -- (0,1.2) node[above] {$y$};
    \draw[red,domain=0.1:1,line width=1pt] plot (\x,{0.2/(\x+0.1)});
    \draw[blue,domain=0.0:1,line width=1pt] plot (\x,{(-5*\x+8)/10});
    \draw[dashed] (0.18211, 0.70895) -- (0.18211, 0) node[below] {$x^*$};
    \fill (0.18211, 0.70895) circle[radius=0.8pt];
    \draw[domain=0.18211:1,line width=2pt] plot (\x,{(-5*\x+8)/10});
    \begin{scope}
      \fill[pattern=north east lines,pattern color=red, opacity=0.5, domain=0.1:1, variable=\x]
        (0.1,1) plot (\x,{0.2/(\x+0.1)}) -- (1,2/11) -- (1,1) -- cycle;
    \end{scope}
    \end{tikzpicture}
    \subcaption{\footnotesize $\conv(H_1)$ and $l$}
\end{minipage}%
\begin{minipage}{0.32\textwidth}
    \centering
    \begin{tikzpicture}[scale=2.5]
    \draw[dashed] (1,0) -- (1,1);
    \draw[dashed] (0,1) -- (1,1);
    \draw[->] (-0.1,0) -- (1.2,0) node[right] {$x$};
    \draw[->] (0,-0.1) -- (0,1.2) node[above] {$y$};
    \draw[darkgreen,domain=0:1-13/45,line width=1pt] plot (\x,{-0.13/(\x-1)+0.55});
    \draw[blue,domain=0.0:1,line width=1pt] plot (\x,{(-5*\x+8)/10});
    \draw[dashed] (0.18211, 0.70895) -- (0.18211, 0) node[below] {$x^*$};
    \fill (0.18211, 0.70895) circle[radius=0.8pt];
    \draw[domain=0:0.2,line width=2pt] plot (\x,{(-5*\x+8)/10});
    \begin{scope}
      \fill[pattern=north east lines,pattern color=darkgreen, opacity=0.5, domain=0:32/45, variable=\x]
        (0,0.68) plot (\x,{-0.13/(\x-1)+0.55}) -- (32/45,1) -- (0,1) -- cycle;
    \end{scope}
    \end{tikzpicture}
    \subcaption{\footnotesize $\conv(H_3)$ and $l$}
\end{minipage}%
\caption{When $\tau > 0$, $m < 0$ and $(m+b)(1-r) \geq \tau$, the branch of hyperbola $H_1$ intersecting with $[0,1]^2$ box is convex and monotonically decreasing whereas the branch of hyperbola $H_3$ intersecting with $[0,1]^2$ box is convex and monotonically increasing. Hence, $\min \{ y \ | \ (x,y) \in \conv(H_3) \}$ is precisely the curve represented by $H_3$ which always lies above the line $l$ for $x > x^*$. }
\label{fig:tau>0_m<0_case1}
\end{figure}
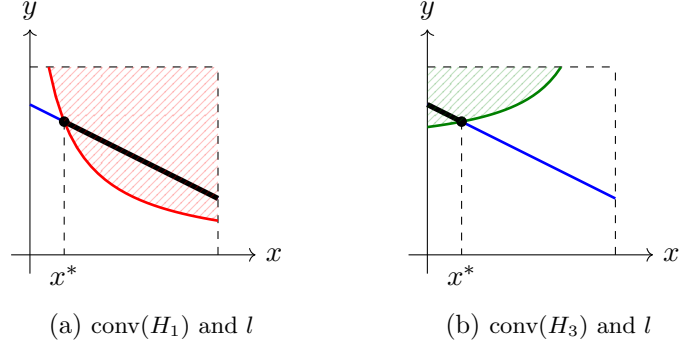

\item $m < 0$ and $(m+b)(1-r) \leq \tau$: We make the following aggregation from \eqref{eq:S_equiv_2} by multiplying $-(1+\epsilon)/m$ to the second equation and adding it to the first equation:
\begin{eqnarray*}
\begin{array}{rrlcl} 
    & & (x- r) y - \tau & = & 0 \\ 
    + & - \frac{1+\epsilon}{m} \quad \times & mx - y + b & = & 0 \\
    = & & \left( x - \tilde r \right) \left( y - \tilde s \right) - \tilde \tau & = & 0
\end{array}
\end{eqnarray*}
where $\tilde r = r - \frac{1+\epsilon}{m}$, $\tilde s = 1+ \epsilon$ and $\tilde \tau = \tau + (1+\epsilon) \left( r - \frac{1+\epsilon-b}{m} \right)$. Again, we can always choose $\epsilon \geq 0$ such that $\tilde \tau \neq 0$. This time, we define $S_3$ and $H_3$ as follows:
$$S_3 = [0,\ 1]^2 \cap H_3, \textup{ where } H_3 = 
\{x, y \in \mathbb{R}^2 \ | \ (x-\tilde r)(y-\tilde s) = \tilde \tau, y < \tilde s\}.$$ 
We note that $\conv(S_1) \cap S_2$ is on the line $l$ along with the direction that $x \geq x^*$ as illustrated in Figure~\ref{fig:n1=n2=1_tau_nonzero_point_case3} if $1-b \geq m (r + \tau)$ which means that the $x$-coordinate of $H_3$ at $y=1$ is greater than or equal to the $x$-coordinate of $l$ at $y=1$. Then, we have:
\begin{align*}
    \tilde \tau = \tau + (1 + \epsilon) \left( r - \frac{1 + \epsilon - b}{m} \right) \leq (1 + \epsilon) \left( \tau + r - \frac{1-b}{m} \right) \leq 0.
\end{align*}
The last inequality follows from $1+ \epsilon > 0$ and $\tau + r \leq \frac{1-b}{m}$.
We claim that $\conv(S_3) \cap S_2$ is on the line $l$ along the direction that $x \geq x^*$ so that $\conv(S_1) \cap S_2 \cap \conv(S_3)$ is the singleton $(x^*,y^*)$.
Equivalently, we would like to show that the half-line:
\begin{eqnarray}
\label{eq:h3_l_intersect_case3}
\{(x,y) \in l\,|\, x < x^* \} \cap \textup{conv}(H_3) = \emptyset.
\end{eqnarray} 
Note that following Observation~\ref{observation:hyperbola_geometry}, only one branch of $H_3$ intersects $[0,1]^2$ box. 
Since we conclude that $\tilde \tau < 0$, \eqref{eq:h3_l_intersect_case3} implied by: 
\begin{eqnarray}
\label{eq:max_h3_line}
mx + b > \textup{max}\{ y\,|\, (x,y) \in \textup{conv}(H_3)\} = \tilde s + \frac{\tilde \tau}{x-\tilde r} \quad \forall \ x < x^*.  
\end{eqnarray}
which holds because $m < 0$ and $\tilde s + \frac{\tilde \tau}{x-\tilde r}$ is monotonically increasing. Figure~\ref{fig:tau>0_m<0_case2} illustrates the proof.

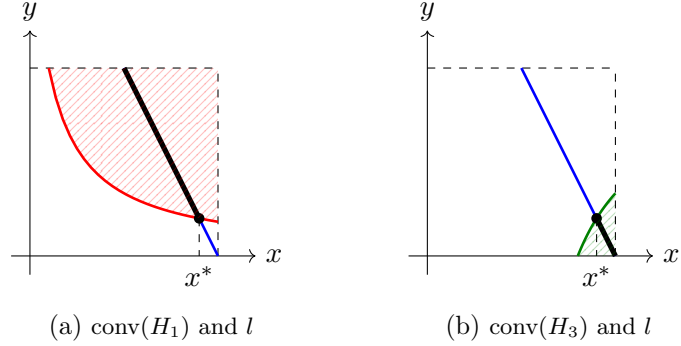
\begin{figure}[hbt!]
\centering
\begin{minipage}{0.32\textwidth}
    \centering
    \begin{tikzpicture}[scale=2.5]
    \draw[dashed] (1,0) -- (1,1);
    \draw[dashed] (0,1) -- (1,1);
    \draw[->] (-0.1,0) -- (1.2,0) node[right] {$x$};
    \draw[->] (0,-0.1) -- (0,1.2) node[above] {$y$};
    \draw[red,domain=0.1:1,line width=1pt] plot (\x,{0.2/(\x+0.1)});
    \draw[blue,domain=0.5:1,line width=1pt] plot (\x,{(-10*\x+10)/5});
    \draw[dashed] (0.9, 0.2) -- (0.9, 0) node[below] {$x^*$};
    \fill (0.9, 0.2) circle[radius=0.8pt];
    \draw[domain=0.5:0.9,line width=2pt] plot (\x,{(-10*\x+10)/5});
    \begin{scope}
      \fill[pattern=north east lines,pattern color=red, opacity=0.5, domain=0.1:1, variable=\x]
        (0.1,1) plot (\x,{0.2/(\x+0.1)}) -- (1,2/11) -- (1,1) -- cycle;
    \end{scope}
    \end{tikzpicture}
    \subcaption{\footnotesize $\conv(H_1)$ and $l$}
\end{minipage}%
\begin{minipage}{0.32\textwidth}
    \centering
    \begin{tikzpicture}[scale=2.5]
    \draw[dashed] (1,0) -- (1,1);
    \draw[dashed] (0,1) -- (1,1);
    \draw[->] (-0.1,0) -- (1.2,0) node[right] {$x$};
    \draw[->] (0,-0.1) -- (0,1.2) node[above] {$y$};
    \draw[darkgreen,domain=0.8:1,line width=1pt] plot (\x,{1-0.4/(\x-0.4)});
    \draw[blue,domain=0.5:1,line width=1pt] plot (\x,{(-10*\x+10)/5});
    \draw[dashed] (0.9, 0.2) -- (0.9, 0) node[below] {$x^*$};
    \fill (0.9, 0.2) circle[radius=0.8pt];
    \draw[domain=0.9:1,line width=2pt] plot (\x,{(-10*\x+10)/5});
    \begin{scope}
      \fill[pattern=north east lines,pattern color=darkgreen, opacity=0.5, domain=0.8:1, variable=\x]
        (0.8,0) plot (\x,{1-0.4/(\x-0.4)}) -- (1,1/3) -- (1,0) -- cycle;
    \end{scope}
    \end{tikzpicture}
    \subcaption{\footnotesize $\conv(H_3)$ and $l$}
\end{minipage}%
\caption{When $\tau > 0$, $m < 0$ and $(m+b)(1-r) \leq \tau$, the branch of hyperbola $H_1$ intersecting with $[0,1]^2$ box is convex and monotonically decreasing whereas the branch of hyperbola $H_3$ intersecting with $[0,1]^2$ box is concave and monotonically increasing. Hence, $\max \{ y \ | \ (x,y) \in \conv(H_3) \}$ is precisely the curve represented by $H_3$ which always lies above the line $l$ for $x > x^*$. }
\label{fig:tau>0_m<0_case2}
\end{figure}
\end{enumerate}
So far, we have only discussed the case $\tau > 0$. When $\tau < 0$, we can make the same argument as above. Note that we can consider $(x-r)y = \tau$ as $(x-r)(-y)=-\tau$ and define $\hat y = - y$ with the new variable bounds or box in $[0,1] \times [-1, 0]$ box.

\paragraph{Proof of \ref{item:tau_zero_line}}
In case $r \in \{0, 1\}$, we see that $\conv(S) = \conv(S_1) \cap S_2$ in Figure~\ref{fig:tau0_r0r1}.
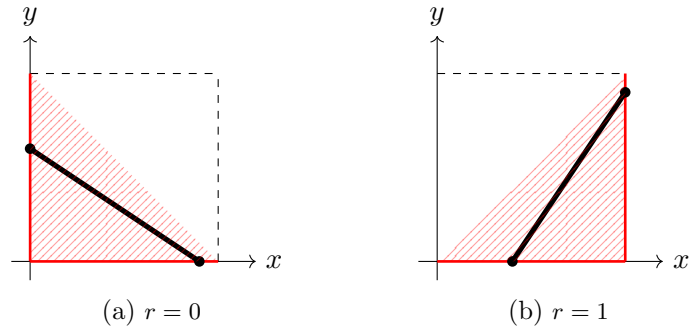
\begin{figure}[hbt!]
\centering
\begin{minipage}{0.32\textwidth}
    \centering
    \begin{tikzpicture}[scale=2.5]
        \draw[dashed] (1,0) -- (1,1);
        \draw[dashed] (0,1) -- (1,1);
        \draw[->] (-0.1,0) -- (1.2,0) node[right] {$x$};
        \draw[->] (0,-0.1) -- (0,1.2) node[above] {$y$};
        \draw[red, line width = 1pt] (0,0) -- (0,1) node[above] {};
        \draw[red, line width = 1pt] (0,0) -- (1,0) node[above] {};
        \draw[black,domain=0:0.9, line width=2pt] plot (\x,{0.6-2/3*\x});
        \fill (0, 0.6) circle[radius=0.8pt];
        \fill (0.9, 0) circle[radius=0.8pt];
        \fill[pattern=north east lines,opacity=0.6, pattern color = red] (0,0) -- (1,0) -- (0,1) -- cycle;
    \end{tikzpicture}   
    \subcaption{\footnotesize $r = 0$}
\end{minipage}
\begin{minipage}{0.32\textwidth}
    \centering
    \begin{tikzpicture}[scale=2.5]
        \draw[dashed] (1,0) -- (1,1);
        \draw[dashed] (0,1) -- (1,1);
        \draw[->] (-0.1,0) -- (1.2,0) node[right] {$x$};
        \draw[->] (0,-0.1) -- (0,1.2) node[above] {$y$};
        \draw[red, line width = 1pt] (1,0) -- (1,1) node[above] {};
        \draw[red, line width = 1pt] (0,0) -- (1,0) node[above] {};
        \draw[black,domain=0.4:1, line width=2pt] plot (\x,{-0.6+3/2*\x});
        \fill (0.4, 0) circle[radius=0.8pt];
        \fill (1, 0.9) circle[radius=0.8pt];
        \fill[pattern=north east lines,opacity=0.6, pattern color = red] (0,0) -- (1,0) -- (1,1) -- cycle;
    \end{tikzpicture}
    \subcaption{\footnotesize $r = 1$}
\end{minipage}
\caption{When $r = 0$ or $r = 1$ and two points intersect between $S_1$ and $S_2$, $\conv(S_1)$ is a triangle that contains $S_2$. Hence the convex hull of $S$ is $S_2$.}
\label{fig:tau0_r0r1}
\end{figure}
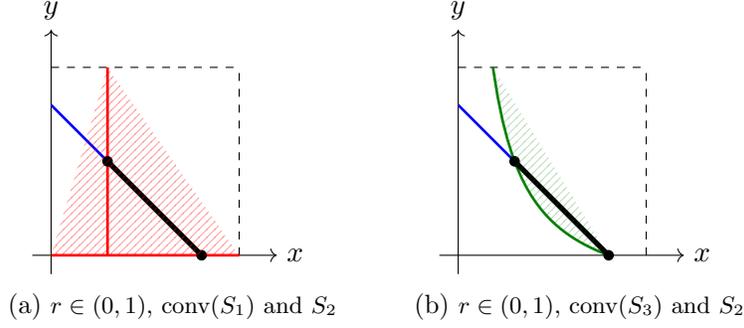
\begin{figure}[hbt!]
\centering
\begin{minipage}{0.32\textwidth}
    \centering
    \begin{tikzpicture}[scale=2.5]
        \draw[dashed] (1,0) -- (1,1);
        \draw[dashed] (0,1) -- (1,1);
        \draw[->] (-0.1,0) -- (1.2,0) node[right] {$x$};
        \draw[->] (0,-0.1) -- (0,1.2) node[above] {$y$};
        \draw[red, line width = 1pt] (0.3,0) -- (0.3,1) node[above] {};
        \draw[red, line width = 1pt] (0,0) -- (1,0) node[above] {};
        \draw[blue,domain=0:0.8, line width=1pt] plot (\x,{0.8-\x});
        \draw[black,domain=0.3:0.8, line width=2pt] plot (\x,{0.8-\x});
        \fill (0.3, 0.5) circle[radius=0.8pt];
        \fill (0.8, 0) circle[radius=0.8pt];
        \fill[pattern=north east lines,opacity=0.6, pattern color = red] (0,0) -- (1,0) -- (0.3,1) -- cycle;
    \end{tikzpicture}
    \subcaption{\footnotesize $r \in (0, 1)$, $\conv(S_1)$ and $S_2$}
\end{minipage}
\begin{minipage}{0.32\textwidth}
    \centering
    \begin{tikzpicture}[scale=2.5]
        \draw[dashed] (1,0) -- (1,1);
        \draw[dashed] (0,1) -- (1,1);
        \draw[->] (-0.1,0) -- (1.2,0) node[right] {$x$};
        \draw[->] (0,-0.1) -- (0,1.2) node[above] {$y$};
        \draw[blue,domain=0:0.8, line width=1pt] plot (\x,{0.8-\x});
        \draw[darkgreen,domain=24/130:0.8, line width=1pt] plot (\x,{-0.3+0.24/(\x)});
        \draw[black,domain=0.3:0.8, line width=2pt] plot (\x,{0.8-\x});
        \fill (0.3, 0.5) circle[radius=0.8pt];
        \fill (0.8, 0) circle[radius=0.8pt];
        \begin{scope}
          \fill[pattern=north east lines,pattern color=darkgreen, opacity=0.5, domain=24/130:0.8, variable=\x]
            (24/130,1) plot (\x,{-0.3+0.24/(\x)}) -- plot (\x,{-13/8*\x + 13/10}) -- (0.8,0) -- cycle;
        \end{scope}
    \end{tikzpicture}
    \subcaption{\footnotesize $r \in (0, 1)$, $\conv(S_3)$ and $S_2$}
\end{minipage}
\caption{When $r \in (0,1)$ and we convert to a hyperbola with $\tilde \tau \neq 0$ as in the right-hand side figure, the same proof can be applied as in the case~\ref{item:tau_notzero_line}.}
\label{fig:tau0_new_hyperbola}
\end{figure}

When $r \in (0,1)$, we can create the following aggregation where it intersects only at one branch and returns to the case similar to \ref{item:tau_notzero_line}.
\begin{eqnarray*}
\begin{array}{rrlcl} 
    & & (x- r) y & = & 0 \\ 
    - & (r + \epsilon) \quad \times & mx - y + b & = & 0 \\
    = & & \left( x - \tilde r \right) \left( y - \tilde s \right) - \tilde \tau & = & 0
\end{array}
\end{eqnarray*}
where $\tilde r = - \epsilon$, $\tilde s = (r + \epsilon) m$ and $\tilde \tau = (r+\epsilon) (b - m \epsilon)$. Again, we can always choose $\epsilon \geq 0$ such that $\tilde \tau \neq 0$ unless both $m = b = 0$. But then, if $m = b = 0$, the equation $mx - y +b = 0$ is equivalent to $y = 0$ and $\conv(S) = \conv(S_1) \cap S_2$. When we have $(x-\tilde r) (y-\tilde s) = \tilde \tau$ form, we note that only one branch of this hyperbola intersects $[0,1]^2$ box and we can apply the same proof as in the case \ref{item:tau_notzero_line}.

\paragraph{Proof of \ref{item:tau_zero_point}}
We consider three aggregations as below:

Aggregation 1:
\begin{eqnarray*}
\begin{array}{rrlcl} 
    & & (x- r) y & = & 0 \\ 
    - & r \quad \times & mx - y + b & = & 0 \\
    = & & x - \left( y - rm\right) & = & rb
\end{array}
\end{eqnarray*}

Aggregation 2:
\begin{eqnarray*}
\begin{array}{rrlcl} 
    & & (x- r) y & = & 0 \\ 
    + & (1 - r) \quad \times & mx - y + b & = & 0 \\
    = & & \left( x - 1 \right) \left( y + (1-r) \right) & = & - (1-r)(b+m)
\end{array}
\end{eqnarray*}

Aggregation 3:
\begin{eqnarray*}
\begin{array}{rrlcl} 
    & & (x- r) y & = & 0 \\ 
    - & \frac{1}{m} \quad \times & mx - y + b & = & 0 \\
    = & & \left( x - \left( r - \frac{1}{m} \right) \right) \left( y - 1 \right) - & = & r + \frac{b+1}{m}.
\end{array}
\end{eqnarray*}
Note that all three aggregations result in hyperbolas that only have one branch intersection in $[0,1]^2$.
Let $\tilde \tau = rb$, $\bar \tau = - (1-r)(b+m)$ and $\hat \tau = r + \frac{b+1}{m}$.
Then note that either $\tilde \tau \neq 0$, $\bar \tau \neq 0$, $\hat \tau \neq 0$ or $\conv(S) = \conv(S_1) \cap S_2$.
Suppose by contradiction that $\tilde \tau = \bar \tau = \hat \tau = 0$. 
Then from $\tilde \tau = 0$, either $r = 0$ or $b = 0$. If $r = 0$ but $b \neq 0$, $\bar \tau = 0$ implies that $m = -b$. $\hat \tau = 0$ then implies that $b = -1$. Together, it implies the line equation $l$ is $x - y - 1 = 0$. But then, this equation intersects with $[0,1]^2$ box only at $(x,y)=(1,1)$; hence, $\conv(S) = \conv(S_1) \cap S_2$. Now if $b = 0$ (irrespective of $r$ value), $\bar \tau = 0$ implies that $m = 0$. The line equation $l$ becomes $y = 0$ which brings us back to the case~\ref{item:tau_zero_line}. Hence, it must be that at least one of $\tilde \tau$, $\bar \tau$ or $\hat \tau$ is nonzero unless we have a trivial case where $S = \{(1,0)\}$ and $\conv(S) = \conv(S_1) \cap S_2$.

Suppose that $\bar \tau \neq 0$. By doing a change of variable, we can give a similar argument on the following set as we have given for the set $S$ from \eqref{eq:S_equiv_2}:
\begin{eqnarray}
\label{eq:S_equiv_3}
S = \left\{ x,y \in [0,1]^2 \, \left | \,  
\begin{array}{lcl} 
(x-1)(y + (1- \bar r)) & = & \bar \tau\\ 
m x - y + b & = & 0
\end{array} \right. \right\}
\end{eqnarray}
where $\bar r = -(1-r)$.
Figure~\ref{fig:tau0_point_subfigure_2} shows the new hyperbola found in the form of $(x-1)(y+(1-\bar r)) = \bar \tau$. Note linear aggregations of the constraint represent the same set as \eqref{eq:S_equiv_2}. Starting from \eqref{eq:S_equiv_3}, we can obtain another set as in Figure~\ref{fig:tau0_point_subfigure_3}.
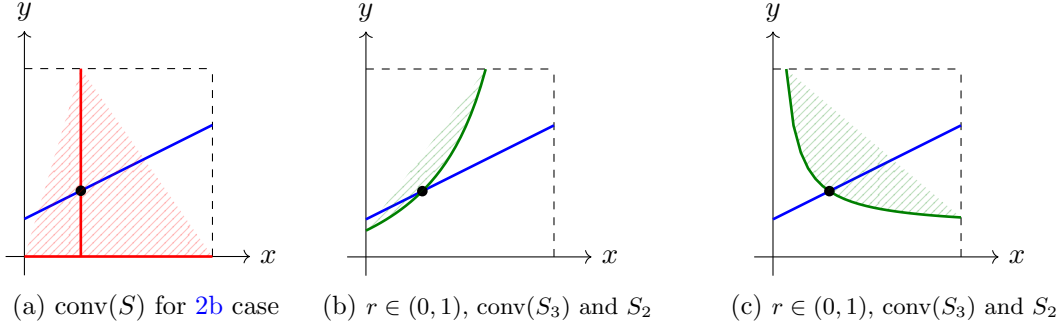
\begin{figure}[hbt!]
\centering
\begin{minipage}{0.23\textwidth}
    \centering
    \begin{tikzpicture}[scale=2.5]
    \draw[dashed] (1,0) -- (1,1);
    \draw[dashed] (0,1) -- (1,1);
    \draw[->] (-0.1,0) -- (1.2,0) node[right] {$x$};
    \draw[->] (0,-0.1) -- (0,1.2) node[above] {$y$};
    \draw[red,line width=1pt] (0,0) -- (1,0);
    \draw[red,line width=1pt] (0.3,0) -- (0.3,1);
    \draw[blue,domain=0:1,line width=1pt] plot (\x,{(\x+0.4)/2});
    \fill (0.3, 0.35) circle[radius=0.8pt];
    \fill[pattern=north east lines,pattern color=red, opacity=0.5] (0,0) -- (0.3,1) -- (1,0) -- cycle;
    \end{tikzpicture}
    \subcaption{$\conv(S)$ for \ref{item:tau_zero_point} case}
    \label{fig:tau0_point_subfigure_1}
\end{minipage}%
\begin{minipage}{0.32\textwidth}
    \centering
    \begin{tikzpicture}[scale=2.5]
        \draw[dashed] (1,0) -- (1,1);
        \draw[dashed] (0,1) -- (1,1);
        \draw[->] (-0.1,0) -- (1.2,0) node[right] {$x$};
        \draw[->] (0,-0.1) -- (0,1.2) node[above] {$y$};
        \draw[blue,domain=0:1,line width=1pt] plot (\x,{(\x+0.4)/2});
        \draw[darkgreen,domain=0:86/135, line width=1pt] plot (\x,{-0.35-0.49/(\x-1)});
        \fill (0.3, 0.35) circle[radius=0.8pt];
        \begin{scope}
          \fill[pattern=north east lines,pattern color=darkgreen, opacity=0.5, domain=0:86/135, variable=\x]
            (0,0.14) plot (\x,{-0.35-0.49/(\x-1)}) -- plot (\x,{1.35*\x + 0.14}) -- (86/135,1) -- cycle;
        \end{scope}
    \end{tikzpicture}
    \subcaption{\footnotesize $r \in (0, 1)$, $\conv(S_3)$ and $S_2$}
    \label{fig:tau0_point_subfigure_2}
\end{minipage}
\begin{minipage}{0.32\textwidth}
    \centering
    \begin{tikzpicture}[scale=2.5]
        \draw[dashed] (1,0) -- (1,1);
        \draw[dashed] (0,1) -- (1,1);
        \draw[->] (-0.1,0) -- (1.2,0) node[right] {$x$};
        \draw[->] (0,-0.1) -- (0,1.2) node[above] {$y$};
        \draw[blue,domain=0:1,line width=1pt] plot (\x,{(\x+0.4)/2});
        \draw[darkgreen,domain=6/85:1, line width=1pt] plot (\x,{0.15+0.06/(\x)});
        \fill (0.3, 0.35) circle[radius=0.8pt];
        \begin{scope}
          \fill[pattern=north east lines,pattern color=darkgreen, opacity=0.5, domain=6/85:1, variable=\x]
            (6/85,1) plot (\x,{0.15+0.06/(\x)}) -- plot (\x,{-0.85*\x + 1.06}) -- (1,0.21) -- cycle;
        \end{scope}
    \end{tikzpicture}
    \subcaption{\footnotesize $r \in (0, 1)$, $\conv(S_3)$ and $S_2$}
    \label{fig:tau0_point_subfigure_3}
\end{minipage}
\caption{When $r \in (0,1)$ and we convert to a hyperbola with $\tilde \tau \neq 0$ as in the right-hand side figure, the same proof can be applied as in the case~\ref{item:tau_notzero_line}.}
\label{fig:tau0_point}
\end{figure}

\section{Proof of Theorem~\ref{theorem:infinite_agg_needed}}
\label{sec:infinite_agg_needed_proof}
Consider the set:
\begin{eqnarray}
\label{eq:finite_agg_counterexample}
S = \left\{ x, y_1, y_2 \in [0, 1]^3\, \left | \,  \begin{array}{rcl} xy_1 &=& 0.5\\ xy_2 &= & 0.5 \end{array} \right. \right\}.
\end{eqnarray}

\begin{observation} Note that since $x \in [0.5, 1]$ for all $(x, y_1, y_2) \in S$, we have that $y_1 = y_2$ for all $(x, y_1, y_2) \in S$. Therefore, $\textup{conv}(S) \subseteq \{ x, y_1, y_2 \in [0, 1]^3\, | \, y_1 = y_2\}$.
\end{observation}

For $\lambda \in \mathbb{R}^2$, let 
$$S_{\lambda} := \left\{ x, y_1, y_2 \in [0, 1]^3\, \left | \,  \lambda_1(xy_1 - 0.5) + \lambda_2 (xy_2 - 0.5) = 0 \right. \right\}.$$
We first show that the intersection of a finite number of convex hulls of $S_\lambda$'s cannot yield the convex hull of $S$ and next show that the intersection of infinitely many convex hulls of $S_\lambda$'s yield the convex hull of $S$. Note first since $(1, 0.5, 0.5), (0.5, 1, 1) \in S$, we have that $\left(\frac{3}{4}, \frac{3}{4}, \frac{3}{4}\right) \in \textup{conv}(S)$.  Moreover, $\left(\frac{3}{4}, \frac{2}{3}, \frac{2}{3}\right) \in S$. Thus, $\left(\frac{3}{4}, \frac{17}{24}, \frac{17}{24} \right) \in \textup{conv}(S)$. Therefore, $\left(\frac{3}{4}, \frac{17}{24}, \frac{17}{24} \right) \in \textup{conv}(S_{\lambda})$ for all $\lambda \in \mathbb{R}^2$.

In order to prove this results, we will show that there exists an $\epsilon_0 >0$ such that 
\begin{eqnarray}\label{eq:0}
\left(\frac{3}{4}, \frac{17}{24} + \epsilon, \frac{17}{24} - \epsilon \right) \in \bigcap_{\lambda \in T} \textup{conv}(S_\lambda) \textup{ for all } 0 \leq \epsilon \leq \epsilon_0.
\end{eqnarray}

Therefore, it is sufficient to prove that for a given value of $\lambda \in \mathbb{R}^2$, there exist $\hat{\epsilon}(\lambda) > 0$, such that  
$$\left(\frac{3}{4}, \frac{17}{24} + \hat{\epsilon}(\lambda), \frac{17}{24} - \hat{\epsilon}(\lambda) \right) \in \textup{conv}(S_\lambda),$$
since we can then take a convex combination with the point $\left(\frac{3}{4}, \frac{17}{24}, \frac{17}{24} \right) \in \textup{conv}(S_{\lambda})$ to obtain points of the form in (\ref{eq:0}). Finally, taking the smallest (still positive) value of $\hat{\epsilon}(\lambda)$ from the finite set $\{\hat{\epsilon}(\lambda)\,|\, \lambda \in T\}$ and setting it to $\epsilon_0$ completes the proof. 

By re-scaling $\lambda$, we have the following two cases to consider:
\begin{enumerate}
\item $\lambda_1 = 0$:  In this case, note that $\left(\frac{3}{4}, \frac{3}{4}, \frac{2}{3} \right) \in S_{\lambda}$.
\item $\lambda_1 = 1$: For convenience, let $\theta := \lambda_2$. We consider the following sub-cases:
\begin{itemize}
\item $\theta \geq \frac{4}{3}$: In this case, note that it can be verified that
\begin{eqnarray}\label{eq:1}
\left(\frac{3}{4}, \frac{17}{24} + \epsilon, \frac{17}{24} - \epsilon \right), 
\end{eqnarray}
satisfies the equation $xy_1 - 0.5 + \theta (xy_2 - 0.5) = 0$, where $\epsilon = \frac{1}{24}\cdot \frac{\theta + 1}{\theta -1}$. Clearly, $\epsilon > 0$. Note that, equivalently, we have $\theta = \frac{24\epsilon + 1}{24\epsilon -1}$. Since $\theta \geq \frac{4}{3}$ we obtain that $\epsilon \leq \frac{7}{24}$ which guarantees that $0 \leq \frac{17}{24} + \epsilon \leq 1$ and $0 \leq \frac{17}{24} - \epsilon \leq 1$. Thus, (\ref{eq:1}) belongs to $S_{\lambda}$.
\item $-1 \leq \theta  < 4/3$: In this case, note that it can be verified that
\begin{eqnarray}\label{eq:2}
\left(\hat{x}(\theta), 1, \frac{10}{24} \right) \equiv \left(\hat{x}(\theta), \frac{17}{24} + \frac{7}{24}, \frac{17}{24} - \frac{7}{24} \right),
\end{eqnarray}
satisfies the equation $xy_1 - 0.5 + \theta (xy_2 - 0.5) = 0$, where $\hat{x}(\theta) = \frac{6\theta + 6}{12 + 5\theta}.$ Note that $0 \leq \hat{x}(\theta) < \frac{3}{4}$ for all $-1 \leq \theta  < \frac{4}{3}$. We now take a convex combination between (\ref{eq:2}) and the point 
\begin{eqnarray}
\frac{5}{12}\cdot\left(\frac{1}{2}, 1,1\right) + \frac{7}{12}\cdot\left(1, \frac{1}{2}, \frac{1}{2} \right)= \left(\frac{19}{24}, \frac{17}{24}, \frac{17}{24} \right) \in \textup{conv}(S) \subseteq \textup{conv}(S_{\lambda}).
\end{eqnarray}
Noting that $\frac{19}{24} > \frac{3}{4} > \hat{x}(\theta)$, we obtain that we can always find a convex combination (with the multiplier to (\ref{eq:2}) being positive) which is of the form $\left(\frac{3}{4}, \frac{17}{24} + \epsilon, \frac{17}{24} - \epsilon \right)$ with $\epsilon= \frac{12+5\theta}{288-84\theta}$. Since $-1 \leq \theta  < 4/3$, $\frac{7}{24\cdot19} < \epsilon \leq \frac{7}{24}$.
\item $\theta < -1$: In this case, we again consider the point (\ref{eq:1}).  Note that $\epsilon = \frac{1}{24}\cdot \frac{\theta + 1}{\theta -1}$. Clearly, $\epsilon > 0$ for $\theta < -1$. Moreover,  $\frac{1}{24}\cdot \frac{\theta + 1}{\theta -1} < \frac{1}{24}$ for all $\theta < -1$. This guarantees that $0 \leq \frac{17}{24} + \epsilon \leq 1$ and $0 \leq \frac{17}{24} - \epsilon \leq 1$. Thus, (\ref{eq:1}) belongs to $S_{\lambda}$.
\end{itemize}
\end{enumerate}

Next, we show that infinite aggregations yield a convex hull. First note that, since $\textup{conv}(S) \subseteq \bigcap_{\lambda \in \mathbb{R}^2} \textup{conv}(S_\lambda)$, we want to show $\bigcap_{\lambda \in \mathbb{R}^2} \textup{conv}(S_\lambda) \subseteq \conv(S)$. We show the result in two parts.  
First, we show that:
\begin{eqnarray}\label{eq:part1}  \bigcap_{\lambda \in \mathbb{R}^2} \conv(S_\lambda) \subseteq 
    \left\{ x, y_1, y_2 \in [0, 1]^3\, \left | \,  
    \begin{array}{l} 
    x (y_1 + y_2) \geq 1 \\
    x + y_1 \leq 1.5 \\
    y_1 = y_2 
    \end{array} 
    \right. \right\},
\end{eqnarray}
and then we show that:
\begin{eqnarray}\label{eq:part2}
    \left\{ x, y_1, y_2 \in [0, 1]^3\, \left | \,  
    \begin{array}{l} 
    x (y_1 + y_2) \geq 1 \\
    x + y_1 \leq 1.5 \\
    y_1 = y_2 
    \end{array} 
    \right. \right\} = \conv(S).
\end{eqnarray}

We begin by proving (\ref{eq:part1}) in three steps.
\begin{claim}
    \begin{align*}
    \bigcap_{\lambda \in \mathbb{R}^2} \conv(S_\lambda) \subseteq \left\{ x, y_1, y_2 \in [0, 1]^3 \, \left | \,  x(y_1 + y_2) \geq 1 \right.\right\}.
\end{align*}
\end{claim}

First, note that:
\begin{align*}
    S_{(1,1)} = \left\{ x, y_1, y_2 \in [0, 1]^3 \, \left | \,  x(y_1 + y_2) = 1 \right.\right\} \subseteq \left\{ x, y_1, y_2 \in [0, 1]^3 \, \left | \,  x(y_1 + y_2) \geq 1 \right.\right\},
\end{align*}
and $\left\{ x, y_1, y_2 \in [0, 1]^3 \, \left | \,  x(y_1 + y_2) \geq 1 \right.\right\}$ is a convex set on the domain of $x, y_1, y_2 \in [0, 1]^3$. Therefore, it follows:
\begin{align*}
    \bigcap_{\lambda \in \mathbb{R}^2} \conv(S_\lambda) \subseteq \conv(S_{(1,1)}) \subseteq \left\{ x, y_1, y_2 \in [0, 1]^3 \, \left | \,  x(y_1 + y_2) \geq 1 \right.\right\},
\end{align*}
where the second inclusion is due to the convexity of $\left\{ x, y_1, y_2 \in [0, 1]^3 \, \left | \,  x(y_1 + y_2) \geq 1 \right.\right\}$.
\begin{claim}
\begin{align*}
    \bigcap_{\lambda \in \mathbb{R}^2} \conv(S_\lambda) \subseteq \left\{x, y_1, y_2 \in [0, 1]^3 \, \left | \, y_1 = y_2 \right. \right\}
\end{align*}
\end{claim}
\begin{proof}
Let $\lambda_1 = 1$ and $\lambda_2 = - \theta$:
\begin{align*}
    S_\lambda & = \left\{x, y_1, y_2 \in [0, 1]^3 \, \left | \, (xy_1 - 0.5) - \theta (xy_2 - 0.5) = 0\right. \right\}.
\end{align*}
If $x = 0$ and $(x,y_1, y_2) \in S_{\lambda}$,  then we must have that $\theta = 1$. In other words, if $\theta \neq 1$, then $0 \notin \textup{proj}_x(S_{\lambda})$.
Therefore, for $\theta \neq 1$, the bilinear constraint defining $S_{\lambda}$ can be reformulated as:
\begin{align*}
    y_1 - \theta y_2 = \frac{0.5(1-\theta)}{x}.
\end{align*}
We now consider two cases:
\begin{enumerate}
    \item $\theta < 1$: In this case note that $y_1 - \theta y_2 = 0.5 (1- \theta) / x \geq 0.5(1-\theta)$ and
    \begin{align*}
        S_\lambda \subseteq H_{\lambda}^+ := \left\{x, y_1, y_2 \in [0, 1]^3 \, \left | \, (y_1 - 0.5) - \theta (y_2 - 0.5) \geq 0\right. \right\}.
    \end{align*}
    \item $\theta > 1$: In this case note that $y_1 - \theta y_2 = 0.5 (1- \theta) / x \leq 0.5(1-\theta)$ and
    \begin{align*}
        S_\lambda \subseteq H_\lambda^- := \left\{x, y_1, y_2 \in [0, 1]^3 \, \left | \, (y_1 - 0.5) - \theta (y_2 - 0.5) \leq 0\right. \right\}.
    \end{align*}
\end{enumerate}
Since $H_\lambda^+$ and $H_\lambda^-$ are both convex, $\conv(S_\lambda) \subseteq H_\lambda^+$ for $\theta < 1$ and $\conv(S_\lambda) \subseteq H_\lambda^-$ for $\theta > 1$. Then we have:
\begin{align*}
    \bigcap_{\lambda \in \mathbb{R}^2} \conv(S_\lambda) 
    & \subseteq \left( \bigcap_{\theta < 1} H_\lambda^+ \right) \bigcap \left( \bigcap_{\theta > 1} H_\lambda^- \right) \\
    & = \left\{x, y_1, y_2 \in [0, 1]^3 \, \left | \, y_1 = y_2, \; y_1, y_2 \geq 0.5 \right. \right\}.
\end{align*}
Figure~\ref{fig:thetaconverge} illustrates the proof.
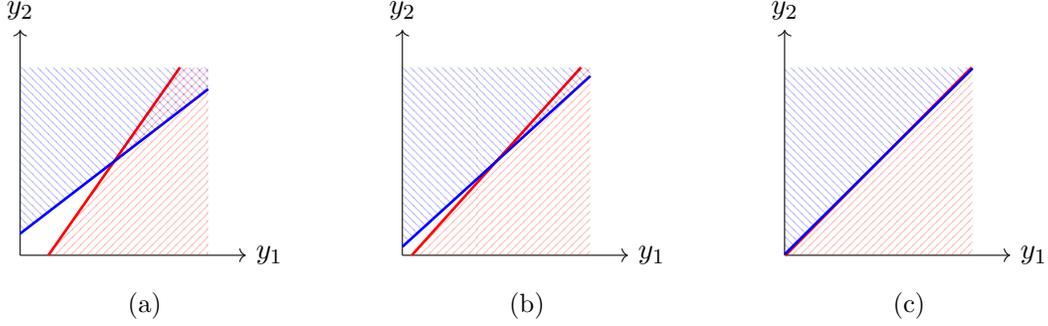
\begin{figure}[tbh!]
    \centering
    \begin{minipage}{0.3\textwidth}
        \centering
        \begin{tikzpicture}[scale=2.5]
        \draw[->] (0,0) -- (1.2,0) node[right] {$y_1$};
        \draw[->] (0,0) -- (0,1.2) node[above] {$y_2$};
        
        \draw[red,domain=0.15:0.85, line width=1pt] plot (\x,{(\x-0.15)/0.7});
        \draw[blue,domain=0:1, line width=1pt] plot (\x,{(\x+0.15)/1.3});

        \fill[pattern=north east lines,pattern color=red, opacity=0.5] (0.15, 0) -- (0.85, 1) -- (1,1) -- (1, 0) -- cycle;
        \fill[pattern=north west lines,pattern color=blue, opacity=0.5] (0, 15/130) -- (0, 1) -- (1,1) -- (1, 115/130) -- cycle;
\end{tikzpicture}
        \subcaption{}
    \end{minipage}
    \begin{minipage}{0.3\textwidth}
        \centering
        \begin{tikzpicture}[scale=2.5]
        \draw[->] (0,0) -- (1.2,0) node[right] {$y_1$};
        \draw[->] (0,0) -- (0,1.2) node[above] {$y_2$};
        
        \draw[red,domain=0.05:0.95, line width=1pt] plot (\x,{(\x-0.05)/0.9});
        \draw[blue,domain=0:1, line width=1pt] plot (\x,{(\x+0.05)/1.1});

        \fill[pattern=north east lines,pattern color=red, opacity=0.5] (0.05, 0) -- (0.95, 1) -- (1,1) -- (1, 0) -- cycle;
        \fill[pattern=north west lines,pattern color=blue, opacity=0.5] (0, 5/110) -- (0, 1) -- (1,1) -- (1, 105/110) -- cycle;
\end{tikzpicture}
        \subcaption{}
    \end{minipage}
    \begin{minipage}{0.3\textwidth}
        \centering
        \begin{tikzpicture}[scale=2.5]
        \draw[->] (0,0) -- (1.2,0) node[right] {$y_1$};
        \draw[->] (0,0) -- (0,1.2) node[above] {$y_2$};
        
        \draw[red,domain=0.005:0.995, line width=1pt] plot (\x,{(\x-0.005)/0.99});
        \draw[blue,domain=0:1, line width=1pt] plot (\x,{(\x+0.005)/1.01});

        \fill[pattern=north east lines,pattern color=red, opacity=0.5] (0.005, 0) -- (0.995, 1) -- (1,1) -- (1, 0) -- cycle;
        \fill[pattern=north west lines,pattern color=blue, opacity=0.5] (0, 5/1010) -- (0, 1) -- (1,1) -- (1, 1005/1010) -- cycle;
\end{tikzpicture}
        \subcaption{}
    \end{minipage}
    \caption{The intersection becomes a line segment as $\theta \rightarrow 1$ and $\theta \leftarrow 1$.}
    \label{fig:thetaconverge}
    \label{fig:infinite_intersection}
\end{figure}
\end{proof}

The next claim completes the proof of (\ref{eq:part1}).
\begin{claim}
\begin{align*}
    \left( \bigcap_{\lambda \in \mathbb{R}^2} \conv(S_\lambda) \right)  \subseteq \left\{ x, y_1, y_2 \in [0, 1]\, \left | \, x + y_1 \leq 1.5 \right. \right\}
\end{align*}
\end{claim}
\begin{proof}
Observe that 
$$\textup{conv}({S_{(1,0)}}) = \{ x, y_1, y_2 \in [0, 1]^3\,|\,xy_1 \geq 0.5, x + y_1 \leq 1.5\} \subseteq \left\{ x, y_1, y_2 \in [0, 1]\, \left | \, x + y_1 \leq 1.5 \right. \right\}.$$ 
Thus, we obtain the result.
\end{proof}
Next, we verify (\ref{eq:part2}). Let $T \subseteq \mathbb{R}^3$ be the set in the left-hand-side of (\ref{eq:part2}). Note first that verifying $T \supseteq \textup{conv}(S)$ is straightforward as the $T$ is convex and contains $S$.

To prove $T\subseteq \textup{conv}(S)$, we show that any point in the $T$ can be written as the convex combination of two points in $S$. In particular, let $(\hat x, \hat y_1, \hat y_2) \in T$. Then, $\hat x \geq 0.5$. If $\hat x = 0.5$ or $\hat{y}_1 = 1$, then the constraint $\hat x (\hat y_1 + \hat y_2) \geq 1$ with the box constraints on $y_1,y_2 \in [0,1]^2$ implies that 
$(\hat x, \hat y_1, \hat y_2) = (0.5, 1, 1) \in S$. For $\hat x > 0.5$ and $\hat{y}_1 < 1$, we verify that
\begin{eqnarray}\label{eq:convTS}
    (\hat x, \hat y_1, \hat y_2) = \lambda (0.5, 1, 1) + (1- \lambda)(\tilde x, \tilde y_1, \tilde y_2), 
\end{eqnarray} 
where
\begin{align*}
    & \left( \tilde x, \tilde y_1, \tilde y_2 \right) = \left( \frac{\hat x-0.5}{1- \hat y_1}, \frac{1- \hat y_1}{2 \hat x-1},\frac{1- \hat y_1}{2 \hat x-1} \right), \quad \lambda = \frac{2\hat x \hat y_1 - 1}{2 \hat x + \hat y_1 -2}.
\end{align*}
We first verify that $(\tilde x, \tilde y_1, \tilde y_2) \in S$.
\begin{itemize}
    \item It is straightforward that $\tilde x \tilde y_1 = 0.5$ and $\tilde x \tilde y_2 = 0.5$.
    \item $\tilde x, \tilde y_1, \tilde y_2 \geq 0$ because $\hat{x} > 0.5$ and $\hat{y}_1 < 1$.
    \item $\tilde x \leq 1$ because $\hat x + \hat y_1 \leq 1.5$ which implies $\hat x - 0.5 \leq 1 - \hat y_1$.
    \item $\tilde{y}_1, \tilde{y}_2 \leq 1$ is equivalent to $2\hat{x} + \hat{y}_1 \geq 2.$ This is true because $\textup{min}\{2\hat{x} + \hat{y}_1\,|\, \hat{x} \hat{y}_1 \geq 0.5,\ \hat{x}, \hat{y}_1 \in [0, 1] \}  =2,$ where the constraint $\hat{x}\hat{y}_1 \geq 0.5$ is implied by the constraints $\hat{x} (\hat{y}_1 + \hat{y}_2) \geq 1$ and $\hat{y}_1 = \hat{y}_2$ defining $T$.
\end{itemize}
Also, since $(\hat x, \hat y_1, \hat y_2) \in T$ implies that $2\hat{x}\hat{y}_1\geq 1$ and $(2\hat{x}\hat{y}_1  -1)(1 - \hat{y}) \geq 1$, we have that $0 \leq \lambda \leq 1$. 
it remains to verify (\ref{eq:convTS}):
\begin{itemize}
    \item 
    $     
    \begin{aligned}[t]
        0.5\lambda + (1 - \lambda) \tilde x & = \frac{\hat x \hat y_1 -0.5}{2\hat x + \hat y_1 - 2} + \frac{2\hat x + \hat y_1 - 2 \hat x \hat y_1 - 1}{2 \hat x + \hat y_1 - 2} \cdot \frac{\hat x - 0.5}{1 - \hat y_1} \\
        & = \frac{\hat x \hat y_1 -0.5}{2\hat x + \hat y_1 - 2} + \frac{(2\hat x - 1)(1 - \hat y_1)}{2 \hat x + \hat y_1 - 2} \cdot \frac{\hat x - 0.5}{1 - \hat y_1} \\
        & = \frac{\hat x \hat y_1 -0.5 + 2(\hat x - 0.5)^2}{2\hat x + \hat y_1 - 2} = \frac{2 \hat x^2 + \hat x \hat y_1 - 2 \hat x}{2\hat x + \hat y_1 - 2} = \hat x
    \end{aligned}  
    $
    \item 
    $     
    \begin{aligned}[t]
        \lambda + (1-\lambda) \tilde y_1 & = \frac{2\hat x \hat y_1 - 1}{2 \hat x + \hat y_1 -2} + \frac{2\hat x + \hat y_1 - 2 \hat x \hat y_1 - 1}{2 \hat x + \hat y_1 - 2} \cdot \frac{1 - \hat y_1}{2\hat x - 1} \\
        & = \frac{2\hat x \hat y_1 - 1 + (1-\hat y_1)^2}{2 \hat x + \hat y_1 - 2} = \frac{2\hat x \hat y + \hat y_1^2 - 2 \hat y_1}{2 \hat x + \hat y_1 - 2} = \hat y_1
    \end{aligned}  
    $
    \item 
    $     
    \begin{aligned}[t]
        \lambda + (1-\lambda) \tilde y_2 = \lambda + (1-\lambda) \tilde y_1 = \hat y_1 = \hat y_2
    \end{aligned}  
    $
\end{itemize}
This completes the proof that infinite aggregations are needed for certain sets with $n_1+n_2 \geq 3$.

\section{Proof of Theorem~\ref{theorem:agg_no_work}}
\label{sec:agg_no_work_proof}

Consider the set:
\begin{eqnarray*}
S = \left\{ x, y_1, y_2 \in [0, 1]^3\, \left | \,  
\begin{array}{lcl} 
-2 x y_1 + 9 x y_2 + y_1 -5 y_2 &=& 0 \\ 
5 x y_1 + 3 y_1 + 3 y_2 &= & 6 
\end{array} \right. \right\}.
\end{eqnarray*}
For $\lambda \in \mathbb{R}^2$, let
\begin{align*}
    S_\lambda := \{x, y_1, y_2 \in [0, 1] \left | \ \lambda_1 (-2 x y_1 + 9 x y_2 + y_1 -5 y_2 ) + \lambda_2 (5 x y_1 + 3 y_1 + 3 y_2 - 6) = 0 \right \}.
\end{align*}

We consider two cases: (i) $\lambda_1 = 0$ and $\lambda_2 = 1$; and (ii) $\lambda_1 = 1$ and $\theta := \lambda_2$. For both cases, we show that $\hat p = (\hat x, \hat y_1, \hat y_2) = \left(\frac{7}{10}, \frac{7}{8}, \frac{1}{6}\right) \in \conv(S_\lambda)$ but $\hat p \notin \conv(S)$.
First of all, $\hat p \notin S$, because we can find the following separating hyperplane $H$ that separates $S$ and $\hat p$:
\begin{align*}
    H := \left\{ x, y_1, y_2 \in [0, 1] \; | \; -2 x + 10 y_1 - 10 y_2 = 5 \right\}.
\end{align*}
We also define two half spaces:
\begin{align*}
    & H^{>} := \left\{ x, y_1, y_2 \in [0, 1] \; | \; -2 x + 10 y_1 - 10 y_2 > 5 \right\}, \\
    & H^{<} := \left\{ x, y_1, y_2 \in [0, 1] \; | \; -2 x + 10 y_1 - 10 y_2 < 5 \right\}.
\end{align*}
Since $- 2 \cdot \frac{7}{10} + 10 \cdot \frac{7}{8} - 10 \cdot \frac{1}{6} = \frac{341}{60} > 5$, $\hat p \in H^{>}$. Next we want to show that for all $p \in S$, $p \in H^{<}$, so that $\conv(S) \in H^{<}$.
Note that we can write $y_1$ and $y_2$ in terms of $x$ by solving two systems of equations:
\begin{align*}
    y_1 = \frac{54 x - 30}{45 x^2 + 8x - 18}, \quad y_2 = \frac{12 x - 6}{45 x^2 + 8x - 18}
\end{align*}
Let the denominator function be $f(x) = 45 x^2 + 8x - 18$. Then, 
$f(x) = 0$ at $\underline{x} = \frac{-4 - \sqrt{826}}{45} \approx -0.7276$ and $\bar x = \frac{\sqrt{826}-4}{45} \approx 0.5498$. 
For any $x \in (\underline{x}, \bar x)$, $f(x) < 0$ and $f(x) > 0$ if $x < \underline{x}$ or $x > \bar x$.
For boundary constraints of $y_1, y_2 \in [0,1]$ to be satisfied, 
we refer to the graph of $y_1$ and $y_2$ in Figure~\ref{fig:counterexample_y1y2_graphs}.
From this, we note that the only interval of $x \in [0,1]$ such that both $y_1, y_2 \in [0,1]$ is when $x \geq \frac{2 + 4\sqrt{34}}{45} \approx 0.5628$.
\begin{figure}[tbh!]
\centering
    \begin{minipage}{0.45\textwidth}
    \centering
    \begin{tikzpicture}[scale=3]
        \draw[dashed] (1,0) -- (1,1);
        \draw[dashed] (0,1) -- (1,1);
        \draw[->] (-0.1,0) -- (1.2,0) node[right] {$x$};
        \draw[->] (0,-0.1) -- (0,1.1) node[left] {$y_1$};
        \draw[black,domain=5/9:1,line width=1pt] plot (\x,{(54*\x - 30)/(-18 + 8*\x + 45*(\x)^2});
    \end{tikzpicture}   
    \subcaption{$y_1$}
    \label{fig:counterexample_y1}
    \end{minipage}
    \begin{minipage}{0.45\textwidth}
    \centering
    \begin{tikzpicture}[scale=3]
        \draw[dashed] (1,0) -- (1,1);
        \draw[dashed] (0,1) -- (1,1);
        \draw[->] (-0.1,0) -- (1.2,0) node[right] {$x$};
        \draw[->] (0,-0.1) -- (0,1.1) node[left] {$y_2$};
        \draw[black,domain=0:1/2,line width=1pt] plot (\x,{(12*\x - 6)/(-18 + 8*\x + 45*(\x)^2});
        \draw[black,domain=0.562751:1,line width=1pt] plot (\x,{(12*\x - 6)/(-18 + 8*\x + 45*(\x)^2});
    \end{tikzpicture}   
    \subcaption{$y_2$}
    \label{fig:counterexample_y2}
    \end{minipage}
    \caption{For $y_1$, the boundary constraint of $y_1 \in [0,1]$ is satisfied when $x \geq \frac{5}{9}$ whereas for $y_2$, the boundary constraint is satisfied when either $x \leq \frac{1}{2}$ or $x \geq \frac{2 + 4\sqrt{34}}{45}$. These conditions together require $x \geq \frac{2 + 4\sqrt{34}}{45}$.}
    \label{fig:counterexample_y1y2_graphs}
\end{figure}
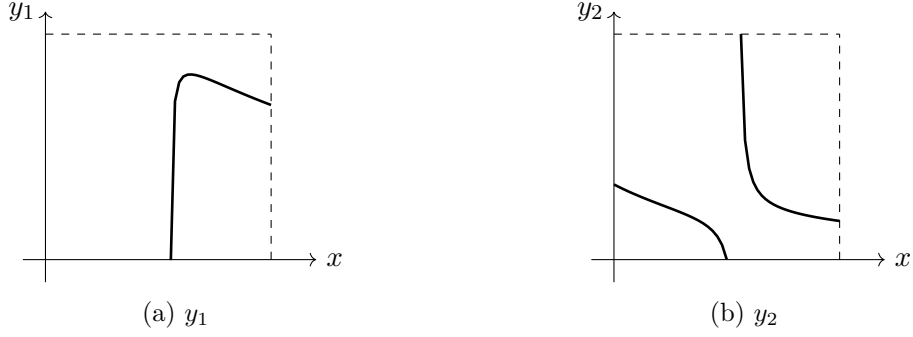
Writing the hyperplane equation in terms of $x$ gives:
\begin{align*}
    -2 x + 10 y_1 - 10 y_2 = \frac{240 - 456 x + 16 x^2 + 90 x^3}{18 - 8 x - 45 x^2} = g(x)
\end{align*}
Taking the derivative of $g(x)$ and setting $g'(x)=0$ gives the places of local minimum and local maximum. Among $x$ values that take a local minimum or local maximum, we only consider values where $x \in \left[\frac{2 + 4\sqrt{34}}{45}, 1\right]$. The only value is the local maximum at $x \approx 0.704$ which gives $g(x) \approx 4.197 < 5$. 
At points $x = \frac{2 + 4\sqrt{34}}{45}$ and $x = 1$, $g(x) \approx -5.9653$ and $g(x) \approx 3.1429$ respectively which are both less than $5$.

Now, consider the following cases to show that $\hat p$ can be obtained from a convex combination of points in $S_\lambda$ for any choice of $\lambda$. In other words, there exists $p_i \in S_\lambda$ and $w_i \in [0,1]$ for $i=1,...,k$ with $k \leq 4$ such that: (i) $\sum w_i = 1$; and (ii) $\sum w_i p_i = \hat p = (\hat x, \hat y_1, \hat y_2) = (\frac{7}{10}, \frac{7}{8}, \frac{1}{6}$).
\begin{enumerate}
    \item $\lambda_1 = 0$ and $\lambda_2 = 1$:
    
    Let $p_1 = \left(1, \frac{1}{2}, \frac{2}{3} \right) $, $p_2 = \left(\frac{3}{5}, 1, 0 \right)$, $w_1 = \frac{1}{4}$ and $w_2 = \frac{3}{4}$.
    \begin{enumerate}
        \item $\sum w_i = \frac{1}{4} + \frac{3}{4} = 1$. 
        \item $\sum w_i p_i = \frac{1}{4} \cdot \left(1, \frac{1}{2}, \frac{2}{3} \right) + \frac{3}{4} \cdot \left(\frac{3}{5}, 1, 0 \right) = \left( \frac{1}{4} + \frac{9}{20}, \frac{1}{8} + \frac{3}{4}, \frac{1}{6}\right) = \left(\frac{7}{10}, \frac{7}{8}, \frac{1}{6}\right)$.
    \end{enumerate}
    \item $\lambda_1 = 1$ and $\lambda_2 = \theta \leq - \frac{5}{3}$: 
    
    Let $
    \begin{aligned}[t]
        & p_1 = \left(0, \tfrac{3\theta+5}{3\theta+1}, 1\right), \; 
        p_2 = \left(1, \tfrac{6\theta}{8\theta-1},0\right), \;
        p_3 = \left(1, \tfrac{3\theta-4}{8\theta-1},1\right), \;
        p_3 = \left(\tfrac{3\theta-1}{5\theta-2},1,0\right), \; \text{with} \\
        & w_1 = \tfrac{47(1+3\theta)}{120(1+47 \theta)}, \;
        w_2 = \tfrac{122+1343\theta-1645\theta^2}{120(1+47 \theta)(1-2\theta)}, \;
        w_3 = \tfrac{799\theta-27}{120(1+47 \theta)}, \;
        w_4 = \tfrac{-11(1-141\theta)(2-5\theta)}{120(1+47 \theta)(1-2\theta)}.
    \end{aligned}
    $
    \begin{enumerate}
        \item $
        \begin{aligned}[t]
        & \sum w_i 
        = (w_1+ w_3) + (w_2 + w_4)
        = \tfrac{940 \theta +20}{120 (1+47\theta)} + \tfrac{100 + 4500 \theta - 9400 \theta^2}{120(1+47 \theta)(1-2\theta)}
        = \tfrac{1}{6} + \tfrac{5}{6} = 1.
        \end{aligned}
        $
        \item 
        $
        \begin{aligned}[t]
            \hat x & = w_2 + w_3 + \tfrac{3\theta-1}{5\theta-2} \cdot w_4\\
            & = \tfrac{122+1343\theta-1645\theta^2}{120(1+47 \theta)(1-2\theta)} + \tfrac{799\theta-27}{120(1+47 \theta)} + \tfrac{3\theta-1}{5\theta-2} \cdot \tfrac{-22+3157 \theta - 7755 \theta^2}{120(1+47 \theta)(1-2\theta)} \\
            & = \tfrac{122+1343\theta-1645\theta^2}{120(1+47 \theta)(1-2\theta)} + \tfrac{-27 + 853 \theta - 1598 \theta^2}{120(1+47 \theta)(1-2\theta)} + \tfrac{-11 + 1584 \theta - 4653 \theta^2}{120(1+47 \theta)(1-2\theta)} \\
            & = \tfrac{84 + 3780 \theta - 7896 \theta^2}{120(1+47 \theta)(1-2\theta)} = \tfrac{84(1+47 \theta)(1-2\theta)}{120(1+47 \theta)(1-2\theta)} = \tfrac{84}{120} = \tfrac{7}{10}.
        \end{aligned}
        $
        \item $
        \begin{aligned}[t]
            \hat y_1 & = \tfrac{3\theta+5}{3\theta+1} \cdot w_1 + \tfrac{6\theta}{8\theta-1} \cdot w_2 + \tfrac{3\theta-4}{8\theta-1} \cdot w_3 + w_4 \\
            & = \tfrac{-78960 \theta^3 + 47670 \theta^2 - 3885 \theta - 105}{120(1+47 \theta)(1-2\theta)(8\theta-1)} =\tfrac{105(1+47\theta)(1-2\theta)(8\theta-1)}{120(1+47 \theta)(1-2\theta)(8\theta-1)} = \tfrac{105}{120} = \tfrac{7}{8}.
        \end{aligned}
        $
        \item $
        \begin{aligned}[t]
            \hat y_2 & = w_1 + w_3 = \tfrac{940\theta+20}{120(1+47\theta)} = \tfrac{20}{120} = \tfrac{1}{6}.
        \end{aligned}
        $
    \end{enumerate}
    \item $\lambda_1 = 1$ $\lambda_2 = \theta \in \left[ -\tfrac{5}{3}, - \tfrac{3}{5} \right]$: 
    
    Let $
    \begin{aligned}[t]
        & p_1 = \left(\tfrac{3\theta+5}{9}, 0, 1\right), \;
          p_2 = \left(1, \tfrac{6\theta}{8\theta-1},0\right),  \;
          p_3 = \left(1, \tfrac{3\theta-4}{8\theta-1},1\right), \;
          p_4 = \left(\tfrac{3\theta-1}{5\theta-2},1,0\right), \; \text{with}\\
        & w_1 = \tfrac{141}{40 (3 \theta - 4) (5 \theta - 11)}, \;
          w_2 = \tfrac{402-595\theta+100\theta^2}{120(2 \theta - 1) (5 \theta - 11)}, \;
          w_3 = \tfrac{457-1060\theta+300\theta^2}{120(3 \theta - 4) (5 \theta - 11)}, \;
          w_4 = \tfrac{(5 \theta - 2) (180 \theta - 349)}{120 (2 \theta - 1) (5 \theta - 11)}.
    \end{aligned}
    $
    \begin{enumerate}
        \item
        $     
        \begin{aligned}[t]
            \sum w_i & = w_1 + w_3 + w_2 + w_4 \\
            & =  \tfrac{3 \cdot 141 + 457-1060\theta+300\theta^2}{120(44-53\theta+15\theta^2)} + \tfrac{402-595\theta+100\theta^2 + 698 - 2105 \theta + 900 \theta^2}{120 (11 - 27 \theta + 10 \theta^2)} \\
            & =  \tfrac{880 -1060\theta+300\theta^2}{120(44-53\theta+15\theta^2)} + \tfrac{1100 -2700 \theta + 1000 \theta^2}{120 (11 - 27 \theta + 10 \theta^2)} = \tfrac{1}{6} + \tfrac{5}{6} = 1
        \end{aligned}  
        $
        \item 
        $     
        \begin{aligned}[t]
            \hat x & = \tfrac{3 \theta + 5}{9} \cdot w_1 + w_2 + w_3 + \tfrac{3 \theta - 1}{5 \theta - 2} \cdot w_4 \\
            & = \tfrac{141 \theta + 235}{120 (3 \theta - 4)(5 \theta - 11)} + \tfrac{402 - 595 \theta + 100 \theta^2}{120 (2 \theta - 1)(5 \theta - 11)} + \tfrac{457-1060\theta+300\theta^2}{120(3 \theta - 4) (5 \theta - 11)} + \tfrac{(3\theta-1)(180\theta-349)}{120(2\theta-1)(5\theta-11)} \\
            & = \tfrac{692 - 919 \theta + 300 \theta^2}{120 (3 \theta - 4)(5 \theta - 11)} + \tfrac{640\theta^2 - 1822\theta + 751}{120 (2 \theta - 1)(5 \theta - 11)} \\
            & = \tfrac{(3\theta - 4) (100\theta - 173)}{120 (3 \theta - 4)(5 \theta - 11)} +\tfrac{(2\theta - 1) (320\theta - 751)}{120 (2 \theta - 1)(5 \theta - 11)} = \tfrac{420\theta -924}{120(5 \theta - 11)} = \tfrac{84(5\theta-11)}{120(5\theta-11)} = \tfrac{7}{10}
        \end{aligned}  
        $
        \item 
        $     
        \begin{aligned}[t]
            \hat y_1 & = \tfrac{6\theta}{8\theta-1} w_2 + \tfrac{3\theta-4}{8\theta-1} w_3 + w_4 \\
            & = \tfrac{6\theta (402 - 595\theta + 100\theta^2)}{120(8\theta-1)(2\theta-1)(5\theta-11)} + \tfrac{457 - 1060\theta + 300\theta^2}{120(8\theta-1)(5\theta-11)} + \tfrac{698-2105\theta+900\theta^2}{120(2\theta-1)(5\theta-11)}\\
            & = \tfrac{600\theta^3 - 3570\theta^2 + 2412\theta}{120(8\theta-1)(2\theta-1)(5\theta-11)} + \tfrac{600\theta^3 - 2420\theta^2 + 1974\theta - 457}{120(8\theta-1)(2\theta-1)(5\theta-11)} + \tfrac{7200\theta^3 - 17740\theta^2 + 7689\theta - 698}{120(8\theta-1)(2\theta-1)(5\theta-11)} \\
            & = \tfrac{-1155 + 12075\theta - 23730\theta^2 + 8400\theta^3 }{120(8\theta-1)(2\theta-1)(5\theta-11)} = \tfrac{105(8\theta-1)(2\theta-1)(5\theta-11)}{120(8\theta-1)(2\theta-1)(5\theta-11)} = \tfrac{7}{8}
        \end{aligned}  
        $
        \item 
        $     
        \begin{aligned}[t]
            \hat y_2 = w_1 + w_3 = \tfrac{3 \cdot 141 + 457-1060\theta+300\theta^2}{120(3\theta-4)(5\theta-11)} = \tfrac{880 - 1060\theta + 300\theta^2}{120(3\theta-4)(5\theta-11)} = \tfrac{20(3\theta-4)(5\theta-11)}{120(3\theta-4)(5\theta-11)} = \tfrac{1}{6}
        \end{aligned}  
        $
    \end{enumerate}
    \item $\lambda_1 = 1$ $\lambda_2 = \theta \in \left[ -\tfrac{3}{5}, - \tfrac{211}{665} \right]$:

    Let
    $
    \begin{aligned}[t]
        & p_1 = \left(\tfrac{3\theta+5}{9}, 0, 1\right), \;
          p_2 = \left(1, \tfrac{6\theta}{8\theta-1},0\right), \;
          p_3 = \left(\tfrac{4}{5\theta+7}, 1,1\right), \;
          p_4 = \left(\tfrac{3\theta-1}{5\theta-2},1,0\right), \; \text{with} \\
        & w_1 = \tfrac{-3(211+665\theta)}{40(65-556\theta-160\theta^2+75\theta^3)}, \;
          w_2 = \tfrac{(-1+8\theta)(958-785\theta-800\theta^2+375\theta^3)}{40(-1+2\theta)(65-556\theta-160\theta^2+75\theta^3)}, \\
        & w_3 = \tfrac{(7 + 5\theta) (457 - 1060\theta + 300\theta^2)}{120(65-556\theta-160\theta^2+75\theta^3)}, \;
          w_4 = \tfrac{(-2 + 5\theta) (1813 - 17094\theta - 3355\theta^2 + 1200\theta^3)}{120 (-1 + 2\theta) (65 - 556\theta - 160\theta^2 + 75\theta^3)}.
    \end{aligned}  
    $
    \begin{enumerate}
        \item 
        $     
        \begin{aligned}[t]
            \sum w_i & = \tfrac{1899 + 2187\theta - 11970\theta^2}{120 (-1 + 2\theta) (65 - 556\theta - 160\theta^2 + 75\theta^3)} + \tfrac{-2874 + 25347\theta - 16440\theta^2 - 20325\theta^3 + 9000\theta^4}{120 (-1 + 2\theta) (65 - 556\theta - 160\theta^2 + 75\theta^3)}  \\
            & + \tfrac{-3199 + 11533\theta - 7070\theta^2 - 7900\theta^3 + 3000\theta^4}{120 (-1 + 2\theta) (65 - 556\theta - 160\theta^2 + 75\theta^3)} + \tfrac{ - 3626 + 43253\theta - 78760\theta^2 - 19175\theta^3 + 6000\theta^4}{120 (-1 + 2\theta) (65 - 556\theta - 160\theta^2 + 75\theta^3)} \\
            & = \tfrac{-7800 + 82320\theta - 114240\theta^2 - 47400\theta^3 + 18000\theta^4}{120 (-1 + 2\theta) (65 - 556\theta - 160\theta^2 + 75\theta^3)} \\
            & = \tfrac{120 (-1 + 2\theta) (65 - 556\theta - 160\theta^2 + 75\theta^3)}{120 (-1 + 2\theta) (65 - 556\theta - 160\theta^2 + 75\theta^3)} = 1
        \end{aligned}  
        $
        \item 
        $     
        \begin{aligned}[t]
            \hat x & = \tfrac{3\theta+5}{9} w_1 + w_2 + \tfrac{4}{5\theta+7} w_3 + \tfrac{3\theta-1}{5\theta-2} w_4 \\
            & = \tfrac{-1055 - 3958\theta - 1995\theta^2}{120(65 - 556\theta - 160\theta^2 + 75\theta^3)} + \tfrac{-2874 + 25347\theta - 16440\theta^2 - 20325\theta^3 + 9000\theta^4}{120 (-1 + 2\theta) (65 - 556\theta - 160\theta^2 + 75\theta^3)} \\
            & \quad + \tfrac{1828 - 4240\theta + 1200\theta^2}{120(65-556\theta-160\theta^2+75\theta^3)} + \tfrac{-1813 + 22533\theta - 47927\theta^2 - 11265\theta^3 + 3600\theta^4}{120(-1+2\theta)(65-556\theta-160\theta^2+75\theta^3)} \\
            & = \tfrac{773 - 8198\theta - 795\theta^2}{120(65-556\theta-160\theta^2+75\theta^3)} + \tfrac{-4687 + 47880\theta - 64367\theta^2 - 31590\theta^3 + 12600\theta^4}{120(-1+2\theta)(65-556\theta-160\theta^2+75\theta^3)} \\
            & = \tfrac{773 - 8198\theta - 795\theta^2}{120(65-556\theta-160\theta^2+75\theta^3)} + \tfrac{(-1 + 2\theta) (4687 - 38506\theta - 12645\theta^2 + 6300\theta^3)}{120(-1+2\theta)(65-556\theta-160\theta^2+75\theta^3)} \\
            & = \tfrac{5460 - 46704\theta - 13440\theta^2 + 6300\theta^3}{120(65-556\theta-160\theta^2+75\theta^3)} = \tfrac{84(65-556\theta-160\theta^2+75\theta^3)}{120(65-556\theta-160\theta^2+75\theta^3)} = \tfrac{7}{10}
        \end{aligned}  
        $
        \item 
        $     
        \begin{aligned}[t]
            \hat y_1 & = \tfrac{6\theta}{8\theta-1} w_2 + w_3 + w_4 \\
            & = \tfrac{17244\theta - 14130\theta^2 - 14400\theta^3 + 6750\theta^4}{120(-1+2\theta)(65-556\theta-160\theta^2+75\theta^3)} + \tfrac{-3199 + 11533\theta - 7070\theta^2 - 7900\theta^3 + 3000\theta^4}{120(-1+2\theta)(65-556\theta-160\theta^2+75\theta^3)} \\
            & \quad + \tfrac{-3626 + 43253\theta - 78760\theta^2 - 19175\theta^3 + 6000\theta^4}{120(-1+2\theta)(65-556\theta-160\theta^2+75\theta^3)} \\
            & = \tfrac{-6825 + 72030\theta - 99960\theta^2 - 41475\theta^3 + 15750\theta^4}{120(-1+2\theta)(65-556\theta-160\theta^2+75\theta^3)} \\
            & = \tfrac{105(-1+2\theta)(65-556\theta-160\theta^2+75\theta^3)}{120(-1+2\theta)(65-556\theta-160\theta^2+75\theta^3)} = \tfrac{7}{8}
        \end{aligned}  
        $
        \item 
        $     
        \begin{aligned}[t]
            \hat y_2 &= w_1 + w_3 \\
            & = \tfrac{-1899 - 5985\theta}{120(65-556\theta-160\theta^2+75\theta^3)} + \tfrac{3199 - 5135\theta - 3200\theta^2 + 1500\theta^3}{120(65-556\theta-160\theta^2+75\theta^3)} \\
            & = \tfrac{1300 - 11120\theta - 3200\theta^2 + 1500\theta^3}{120(65-556\theta-160\theta^2+75\theta^3)} = \tfrac{20(65-556\theta-160\theta^2+75\theta^3)}{120(65-556\theta-160\theta^2+75\theta^3)} = \tfrac{1}{6}
        \end{aligned}  
        $
    \end{enumerate}
    \item $\lambda_1 = 1$ and $\lambda_2 = \theta \in \left[ - \tfrac{211}{665}, \tfrac{1}{3}\right] \cup \left[ \tfrac{863+\sqrt{7682449}}{4110}, \infty \right]$:  

    Let
    $     
    \begin{aligned}[t]
        & p_1 = \left(\tfrac{4}{5\theta+7}, 1,1\right), \;
          p_2 = \left(1, \tfrac{844+863 \theta-2055 \theta^2}{1249+728 \theta-3405 \theta^2}, \tfrac{-(-1+2\theta)(211+665\theta)}{1249+728 \theta-3405 \theta^2}\right), \;
          p_3 = \left(\tfrac{3\theta-1}{5\theta-2},1,0\right), \; \text{with} \\
        & w_1 = \tfrac{47(7+5 \theta)}{1080(3+5 \theta)}, \;
          w_2 = \tfrac{-1249-728 \theta+3405 \theta^2}{1080(-1+2\theta)(3+5\theta)}, \;
          w_3 = \tfrac{277(-2+5\theta)}{1080(-1+2\theta)}.
    \end{aligned}  $
    \begin{enumerate}
        \item 
        $     
        \begin{aligned}[t]
            \sum w_i & = \tfrac{47(7+5 \theta)}{1080(3+5 \theta)} + \tfrac{-1249-728 \theta+3405 \theta^2}{1080(-1+2\theta)(3+5\theta)} + \tfrac{277(-2+5\theta)}{1080(-1+2\theta)} \\
            & = \tfrac{-329+423\theta+470\theta^2}{1080(-1+2\theta)(3+5\theta)} + \tfrac{-1249-728 \theta+3405 \theta^2}{1080(-1+2\theta)(3+5\theta)} + \tfrac{-1662+1385\theta+6925\theta^2}{1080(-1+2\theta)(3+5\theta)} \\
            & = \tfrac{-3240+1080\theta+10800\theta^2}{1080(-1+2\theta)(3+5\theta)} = \tfrac{1080(-1+2\theta)(3+5\theta)}{1080(-1+2\theta)(3+5\theta)} = 1
        \end{aligned}  
        $
        \item 
        $     
        \begin{aligned}[t]
            \hat x & = \tfrac{4}{5\theta+7} w_1 + w_2 + \tfrac{3\theta-1}{5\theta-2} w_3 \\
            & = \tfrac{188}{1080(3+5\theta)} + \tfrac{-1249-728 \theta+3405 \theta^2}{1080(-1+2\theta)(3+5\theta)} + \tfrac{-277+831\theta}{1080(-1+2\theta)} \\
            & = \tfrac{-188+387\theta}{1080(-1+2\theta)(3+5\theta)} + \tfrac{-1249-728 \theta+3405 \theta^2}{1080(-1+2\theta)(3+5\theta)} + \tfrac{-831+1108 \theta+4155 \theta^2}{1080(-1+2\theta)(3+5\theta)} \\
            & = \tfrac{-2268+756\theta+7560\theta^2}{1080(-1+2\theta)(3+5\theta)} = \tfrac{756(-1+2\theta)(3+5\theta)}{1080(-1+2\theta)(3+5\theta)} = \tfrac{7}{10}
        \end{aligned}  
        $ 
        \item 
        $     
        \begin{aligned}[t]
            \hat y_1 & = w_1 + \tfrac{844+863 \theta-2055 \theta^2}{1249+728 \theta-3405 \theta^2} w_2 + w_3 \\
            & = \tfrac{47(7+5 \theta)}{1080(3+5 \theta)} + \tfrac{-844-863\theta+2055\theta^2}{1080(-1+2\theta)(3+5\theta)} + \tfrac{277(-2+5\theta)}{1080(-1+2\theta)} \\
            & = \tfrac{-329 + 423\theta + 470\theta^2}{1080(-1+2\theta)(3+5\theta)} + \tfrac{-844-863\theta+2055\theta^2}{1080(-1+2\theta)(3+5\theta)} + \tfrac{-1662 + 1385\theta + 6925\theta^2}{1080(-1+2\theta)(3+5\theta)} \\
            & = \tfrac{-2835 + 945\theta + 9450\theta^2}{1080(-1+2\theta)(3+5\theta)} = \tfrac{945(-1+2\theta)(3+5\theta)}{1080(-1+2\theta)(3+5\theta)} = \tfrac{7}{8}
        \end{aligned}  
        $
        \item 
        $     
        \begin{aligned}[t]
            \hat y_2 & = w_1 + \tfrac{-(-1+2\theta)(211+665\theta)}{1249+728 \theta-3405 \theta^2} w_2 = \tfrac{47(7+5 \theta)}{1080(3+5 \theta)} + \tfrac{211+665\theta}{1080(3+5\theta)} = \tfrac{540+900\theta}{1080(3+5\theta)} = \tfrac{180(3+5\theta)}{1080(3+t\theta)} = \tfrac{1}{6}
        \end{aligned}  
        $
    \end{enumerate}
    \item $\lambda_1 = 1$ and $\lambda_2 = \theta \in \left[ \tfrac{1}{3}, \tfrac{863+\sqrt{7682449}}{4110}\right]$:  

    Let
    $     
    \begin{aligned}[t]
        & p_1 = \left(\tfrac{4}{5\theta+7}, 1,1\right), \;
          p_2 = \left(1, 0, \tfrac{6\theta}{3\theta+4} \right), \;
          p_3 = \left(\tfrac{2(976+1137\theta)}{45(68+69\theta)},1,\tfrac{844+863 \theta-2055 \theta^2}{27(4+3\theta)(21+115\theta)}\right), \;\text{with}\\
        & w_1 = \tfrac{(7+5\theta)(-412+1635\theta)}{16(712+6629\theta+5685\theta^2)}, \;
          w_2 = \tfrac{1}{8}, \;
          w_3 = \tfrac{9(68+69\theta)(21+115\theta)}{16(712+6629\theta+5685\theta^2)}.
    \end{aligned}  $
    \begin{enumerate}
        \item 
        $     
        \begin{aligned}[t]
            \sum w_i & = \tfrac{(7+5\theta)(-412+1635\theta)}{16(712+6629\theta+5685\theta^2)} + \tfrac{1}{8} + \tfrac{9(68+69\theta)(21+115\theta)}{16(712+6629\theta+5685\theta^2)} \\
            & = \tfrac{-2884+9385\theta+8175\theta^2}{16(712+6629\theta+5685\theta^2)} + \tfrac{1424+13258\theta+11370\theta^2}{16(712+6629\theta+5685\theta^2)} + \tfrac{12852+83421\theta+71415\theta^2}{16(712+6629\theta+5685\theta^2)} \\
            & = \tfrac{11392+106064\theta+90960\theta^2}{16(712+6629\theta+5685\theta^2)} = \tfrac{16(712+6629\theta+5685\theta^2)}{16(712+6629\theta+5685\theta^2)} = 1
        \end{aligned}  
        $
        \item 
        $     
        \begin{aligned}[t]
            \hat x & = \tfrac{4}{5\theta+7} w_1 + w_2 + \tfrac{2(976+1137\theta)}{45(68+69\theta)}w_3 \\
            & = \tfrac{4(-412+1635\theta)}{16(712+6629\theta+5685\theta^2)} + \tfrac{1}{8} + \tfrac{2(976+1137\theta)(21+115\theta)}{80(712+6629\theta+5685\theta^2)} \\
            & = \tfrac{-8240 + 32700\theta}{80(712+6629\theta+5685\theta^2)} + \tfrac{7120 + 66290\theta + 56850\theta^2}{80(712+6629\theta+5685\theta^2)} + \tfrac{40992 + 272234\theta + 261510\theta^2}{80(712+6629\theta+5685\theta^2)} \\
            & = \tfrac{39872+371224\theta+318360\theta^2}{80(712+6629\theta+5685\theta^2)} = \tfrac{56(712+6629\theta+5685\theta^2)}{80(712+6629\theta+5685\theta^2)} = \tfrac{7}{10}   
        \end{aligned}  
        $ 
        \item 
        $     
        \begin{aligned}[t]
            & \hat y_1= \tfrac{(7+5\theta)(-412+1635\theta)}{16(712+6629\theta+5685\theta^2)} + \tfrac{9(68+69\theta)(21+115\theta)}{16(712+6629\theta+5685\theta^2)} \\
            & = \tfrac{-2884+9385\theta+8175\theta^2}{16(712+6629\theta+5685\theta^2)} + \tfrac{12852+83421\theta+71415\theta^2}{16(712+6629\theta+5685\theta^2)} \\
            & = \tfrac{9968+92806\theta+79590\theta^2}{16(712+6629\theta+5685\theta^2)} = \tfrac{14(712+6629\theta+5685\theta^2)}{16(712+6629\theta+5685\theta^2)} = \tfrac{7}{8}
        \end{aligned}  
        $
        \item
        $     
        \begin{aligned}[t]
            \hat y_2 & = w_1 + \tfrac{6\theta}{3\theta+4} w_2 + \tfrac{844+863 \theta-2055 \theta^2}{27(4+3\theta)(21+115\theta)} w_3 \\
            & = \tfrac{(7+5\theta)(-412+1635\theta)}{16(712+6629\theta+5685\theta^2)} + \tfrac{6\theta}{8(3\theta+4)} + \tfrac{(68+69\theta)(844+863 \theta-2055 \theta^2)}{48(4+3\theta)(712+6629\theta+5685\theta^2)} \\
            & = \tfrac{-34608 + 86664\theta + 182565\theta^2 + 73575\theta^3}{48(4+3\theta)(712+6629\theta+5685\theta^2)} + \tfrac{25632\theta - 238644\theta^2 + 204588\theta^3}{48(4+3\theta)(712+6629\theta+5685\theta^2)} \\
            & \quad + \tfrac{57392 + 116920\theta - 80193\theta^2 - 141795\theta^3}{48(4+3\theta)(712+6629\theta+5685\theta^2)} \\
            & = \tfrac{22784 + 229216\theta + 341016\theta^2 + 136440\theta^3}{48(4+3\theta)(712+6629\theta+5685\theta^2)} = \tfrac{8(4+3\theta)(712+6629\theta+5685\theta^2)}{48(4+3\theta)(712+6629\theta+5685\theta^2)} = \tfrac{1}{6}
        \end{aligned}  
        $
    \end{enumerate}
\end{enumerate}
The proof shows that for any choice of $\lambda = (\lambda_1, \lambda_2) \in \mathbb{R}^2$, $\hat p \in \conv(S_\lambda)$. Hence, $\conv(S)  \subsetneq \bigcap_{\lambda \in \mathbb{R}^2} \conv(S_\lambda)$. This counterexample serves as a counterexample for any $n_1 + n_2 > 2$ as we can extend the variable space.

\section*{Acknowledgement}
We would like to thank Trent E. Schreiber for many useful discussions on understanding the FEM update problem and providing the instances. The authors would like to gratefully acknowledge the support of grant number 2211343 from the NSF CMMI.

\bibliographystyle{plain}
\bibliography{reference}

\newpage

\appendix
\section{Bound Reduction}
\label{app:bound_reduction}
We begin by first solving the original nonconvex problem for 60 seconds using BARON, which gives us an objective value that can be used as a cut on $\delta$.
Next, we build the one-row relaxation. Then, for each variable, we maximize and minimize that particular variable over the one-row relaxation together with the objective cut. 
We then update the one-row relaxation using the tighter bounds. This process is applied iteratively 5 times and results in an average bound reduction are reported in Tables~\ref{tab:bound_reduction_12story} and \ref{tab:bound_reduction_16story}. We note that the bound reduction process significantly reduces the range of $y$ variables overall. For $x$ variables, the geometric mean is substantially smaller than the average (arithmetic mean) reflecting that bounds on most of $x$ variables remain unchanged while bounds on a few $x$ variables are reduced significantly. It is interesting to see how even with such small $y$ variable bounds, this optimization problem remains challenging to solve to optimality.

\begin{table}[tbh!]
\centering
\caption{Bound reduction in percentage (\%) for 12-story data}
\label{tab:bound_reduction_12story}
\begin{tabular}{lrrrrr}
\toprule
 & \multicolumn{2}{c}{$x$ variables} &  & \multicolumn{2}{c}{$y$ variables} \\
\cmidrule{2-3} \cmidrule{5-6}
Instance & Average & Geo mean$^{\dagger}$ &  & Average & Geo mean \\
\midrule
1 & 63.51 & 58.15 &  & 99.29 & 99.29 \\
2 & 43.36 & 16.17 &  & 98.90 & 98.90 \\
3 & 33.98 & 0.45 &  & 98.70 & 98.68 \\
4 & 43.79 & 32.60 &  & 98.86 & 98.85 \\
5 & 26.70 & 0.77 &  & 98.36 & 98.34 \\
6 & 25.63 & 1.08 &  & 98.37 & 98.36 \\
7 & 31.38 & 5.84 &  & 98.58 & 98.57 \\
8 & 30.33 & 0.99 &  & 98.60 & 98.59 \\
9 & 38.75 & 7.40 &  & 98.89 & 98.88 \\
10 & 28.92 & 0.79 &  & 98.52 & 98.50 \\
\bottomrule
\end{tabular}
\\$^{\dagger}$\small{The geometric mean has been shifted by 0.001 to avoid data entry of 0.00.}\\
\end{table}

\begin{table}[tbh!]
\centering
\caption{Bound reduction in percentage (\%) for 16-story data}
\label{tab:bound_reduction_16story}
\begin{tabular}{lrrrrr}
\toprule
 & \multicolumn{2}{c}{$x$ variables} &  & \multicolumn{2}{c}{$y$ variables} \\
\cmidrule{2-3} \cmidrule{5-6}
Instance & Average & Geo mean$^{\dagger}$ &  & Average & Geo mean \\
\midrule
1 & 37.33 & 8.61 &  & 98.40 & 98.39 \\
2 & 58.27 & 31.59 &  & 98.93 & 98.93 \\
3 & 50.68 & 44.12 &  & 98.86 & 98.85 \\
4 & 22.16 & 0.08 &  & 97.81 & 97.79 \\
5 & 36.11 & 2.49 &  & 98.26 & 98.25 \\
6 & 20.56 & 0.10 &  & 97.12 & 97.10 \\
7 & 47.82 & 20.87 &  & 98.60 & 98.60 \\
8 & 26.57 & 0.91 &  & 98.29 & 98.27 \\
9 & 39.74 & 16.31 &  & 98.42 & 98.41 \\
10 & 35.10 & 2.22 &  & 97.98 & 97.97 \\
\bottomrule
\end{tabular}
\\$^{\dagger}$\small{The geometric mean has been shifted by 0.001 to avoid data entry of 0.00.}\\
\end{table}



\end{document}